\documentclass[a4paper,12pt]{amsart}
\usepackage[utf8]{inputenc}
\usepackage{amssymb,amsmath}
\usepackage{mathbbol}
\usepackage{amscd}
\usepackage{calrsfs}
\usepackage{cancel}
\usepackage[all]{xy}
\usepackage[onehalfspacing]{setspace}
\usepackage{enumerate}
\usepackage{graphicx}
\usepackage{url}
\usepackage[a4paper, left=2cm, right=2cm, top=3.5cm, bottom=3cm]{geometry}
\usepackage{tikz-cd}
\usepackage{array}
\usepackage[normalem]{ulem}
\usepackage{verbatim}
\usepackage[stretch=10]{microtype}
\usepackage[colorlinks, linktocpage, citecolor = blue, linkcolor = blue,backref]{hyperref} 
\usepackage{braket}

\usepackage{mathabx,epsfig}

\def\A{\mathbb A}
\def\C{\mathbb C}
\def\F{\mathbb F}

\def\k{\mathbb k}

\def\P{\mathbb P}
\def\bP{\mathbb P}
\def\R{\mathbb R}

\def\Q{\mathbb Q}
\def\N{\mathbb N}
\def\Z{\mathbb Z}

\def\cM{\mathcal M}

\def\chr{{\operatorname{char}}}
\def\Hom{{\operatorname{Hom}}}

\def\Image{{\operatorname{Im}}}

\def\Spec{{\operatorname{Spec}}}

\def\Pic{{\operatorname{Pic}}}
\def\NS{{\operatorname{NS}}}

\def\id{{\operatorname{id}}}

\def\RCdim{{\operatorname{RC-dim}}}

\def\codim{{\operatorname{codim}}}

\def\Aut{{\operatorname{Aut}}}

\def\rk{{\operatorname{rk}}}

\def\chr{{\operatorname{char}}}
\def\Bir{{\operatorname{Bir}}}

\def\Cr{{\operatorname{Cr}}}

\def\Burn{{\operatorname{Burn}}}
\def\BURN{{\mathcal{B}urn}}

\def\ExDiv{{\operatorname{ExDiv}}}

\def\wt{\widetilde}

\def\ol{\overline}

\def\Id{{\mathrm{id}}}
\def\MRC{{\mathrm{MRC}}}

\def\ab{{\mathrm{ab}}}

\theoremstyle{plain}
\newtheorem{dummy}{dummy}[section]
\newtheorem{theorem}[dummy]{Theorem}
\newtheorem{thm}[dummy]{Theorem}
\newtheorem{proposition}[dummy]{Proposition}
\newtheorem{prop}[dummy]{Proposition}
\newtheorem{lemma}[dummy]{Lemma}
\newtheorem{lem}[dummy]{Lemma}

\newtheorem{corollary}[dummy]{Corollary}
\newtheorem{cor}[dummy]{Corollary}

\newtheorem{example}[dummy]{Example}

\newtheorem{definition}[dummy]{Definition}

\newtheorem{conjecture}[dummy]{Conjecture}
\numberwithin{equation}{section}

\theoremstyle{definition}
\newtheorem{Def}[dummy]{Definition}
\newtheorem{remark}[dummy]{Remark}

\newcommand{\ssec}{\subsection}

\newcommand{\ti}[1]{\tilde{#1}}
\newcommand{\ul}{\underline}
\newcommand{\vast}{\bBigg@{4}}
\newcommand{\Vast}{\bBigg@{5}}

\newcommand{\cB}{\mathcal{B}}

\newcommand{\cD}{\mathcal{D}}
\newcommand{\cE}{\mathcal{E}}

\newcommand{\cH}{\mathcal{H}}

\newcommand{\cL}{\mathcal{L}}

\newcommand{\cO}{\mathcal{O}}

\newcommand{\cX}{\mathcal{X}}

\newcommand{\cZ}{\mathcal{Z}}

\newcommand{\cProj}{\mathcal{P}roj \ }

\newcommand{\gS}{\Sigma}

\newcommand{\gb}{\beta}
\newcommand{\gd}{\delta}

\newcommand{\gk}{\kappa}

\newcommand{\go}{\omega}
\newcommand{\gs}{\sigma}

\newcommand{\hor}{\mathrm{hor}}

\newcommand{\Ima}{\mathrm{Im}}

\newcommand{\Mor}{\mathrm{Mor}}

\newcommand{\redu}{\mathrm{red}}

\newcommand{\trdeg}{\mathrm{trdeg}}

\newcommand{\ver}{\mathrm{ver}}

\newcommand{\bss}{\setminus}

\newcommand{\cnec}{\mathrel{:=}}
\newcommand{\dto}{\dashrightarrow}

\newcommand{\xto}[1]{\xrightarrow{ #1 }}

\title{Unboundedness
for motivic invariants 
of birational  automorphisms}

\author{Hsueh-Yung Lin, Evgeny Shinder}

	\address{Department of Mathematics, National Taiwan University, 
	and National Center for Theoretical Sciences,
	Taipei, Taiwan.}
	\email{hsuehyunglin@ntu.edu.tw}

\address{School of Mathematical and Physical Sciences, University of Sheffield,
Hounsfield Road, S3 7RH, UK}
\email{eugene.shinder@gmail.com}

\begin{document}

\raggedbottom

\maketitle

\begin{abstract}
We introduce horizontal and vertical motivic invariants of
birational maps between rational dominant maps and study their basic properties.
As a first application, we show that the (usual) motivic invariants vanish for
birational automorphisms of threefolds over algebraically closed fields of characteristic zero.
On the other hand, we prove that the motivic invariants of the birational automorphism groups of many types of
varieties, including projective spaces of dimension at least four
over a field of characteristic zero, do not form a bounded family, even after extending scalars to the algebraic closure of the field.
For such varieties, we further show that their birational automorphism groups are not generated by maps preserving a conic bundle or a rational surface fibration structure, 
and their abelianizations do not stabilize.
\end{abstract}

\setcounter{tocdepth}{1}
\tableofcontents

\section{Introduction}

We work with motivic invariants of birational maps between algebraic varieties introduced in \cite{BirMot, LSZ20}.
Recall that if $\phi\colon X_1 \dto X_2$ is a birational map between algebraic varieties over a field $\k$, then the motivic invariant of $\phi$ is defined as
\begin{equation}\label{eqn:c-intro}
    c(\phi)
\cnec 
\sum_{E \in \ExDiv(\phi^{-1})}
 [E]
 \ \ -  
 \sum_{D \in \ExDiv(\phi)} [D] \ \ \in \ \ \Burn_*(\k).
\end{equation}
Here, $\ExDiv(\phi)$ is the set of exceptional divisors of $\phi$ and $\Burn_*(\k)$ is the Burnside group \cite{KontsevichTschinkel}, that is the free abelian group generated by the birational isomorphism classes of $\k$-varieties.
This invariant has been generalized to birational maps between orbifolds in 
\cite{KT-orbifolds-c}
and to volume preserving maps
in \cite{CLKT, LoginovZhang}.
In this paper, we generalize the motivic invariant $c(\phi)$ to the relative setting and use it to prove new results about the groups of birational automorphisms.

\subsection{Horizontal and vertical invariants}

Given rational dominant maps $\pi_1\colon X_1 \dto B_1$ and
$\pi_2\colon X_2 \dto B_2$ of algebraic varieties over a field $\k$,
we consider a commutative diagram
\begin{equation}\label{eqn-introsq}
    \xymatrix{
    X_1 \ar@{-->}[d]_{\pi_1} \ar@{-->}[r]^\phi & X_2 \ar@{-->}[d]^{\pi_2}\\
    B_1 \ar@{-->}[r]^\sigma
    & B_2 \\
    }
\end{equation}
with $\phi$ and $\sigma$ birational. 
We regard $(\phi, \sigma)$ (or just $\phi$ since $\sigma$ is uniquely determined by $\phi$) 
as a birational map between $\pi_1$ and $\pi_2$. We define the
horizontal invariant $c_\hor(\phi)$ 
(resp. the vertical invariant $c_\ver(\phi)$) by restricting in \eqref{eqn:c-intro} to divisors
which dominate (resp. do not dominate) the respective bases $B_i$; see Definition~\ref{def-chor1}. 
This generalizes the invariants introduced in~\cite{BirMot, LSZ20}, where $B_1 = B_2 = \Spec(\k)$ and $\sigma = \id$,
in which case $c(\phi) = c_{\hor}(\phi)$. 
In general, the absolute invariant $c(\phi)$ decomposes as the sum
\[
c(\phi) = c_\hor(\phi) + c_\ver(\phi).
\]

We study the properties of these invariants in full generality, without making any assumptions on the singularities of the varieties or on the base field $\k$.
We prove three important vanishing results in the case $\pi_1 = \pi_2$. 
The first one, which we  call Vanishing I (Proposition \ref{prop:vanishingI}), is the vanishing of the horizontal invariants when the relative dimension of $\pi_i$ is at most two. The Vanishing II (Corollary \ref{cor:vanishingII})
and Vanishing III (Corollary \ref{cor:vanishingIII}) 
are for vertical invariants in certain situations.

Separating motivic invariants into horizontal and vertical parts allows for an inductive approach to their computation. For example, when $\chr(\k) = 0$, we can consider the maximal rationally connected (MRC) fibration $X \overset{\pi}\dto B$. 
By the uniqueness of the MRC fibration, every $\phi \in \Bir(X)$ induces a birational self-map of $\pi$.
Using the MRC fibration, horizontal and vertical invariants, and our previous results \cite{BirMot, LSZ20}, we prove the following:

\begin{theorem}[= Theorem \ref{thm-van3fold}] \label{thm:main-intro}
Let $X$ be a $3$-dimensional variety over an algebraically closed field $\k$ of characteristic zero.
Then $c(\Bir(X)) = 0$.
\end{theorem}

Let us explain the context of this theorem.
The corresponding result for surfaces over 
an algebraically closed field is trivial,
and over an arbitrary perfect field $\k$ the vanishing result for surfaces 
was deduced from the Minimal Model Program and Sarkisov link factorization in \cite{LSZ20}, where the study of such questions was initiated.
In dimension $3$, the corresponding result is \emph{false} \cite[\S 3.3]{BirMot} over many nonclosed fields (such as $\k = \Q$), 
but it is known to be true when $\k$ is algebraically closed of characteristic zero 
and $X$ is rationally connected \cite[Proposition 2.6]{BirMot}. 
In dimension at least $4$, the result is also \emph{false}, even over $\k = \C$ \cite[\S 3.4]{BirMot}. 

In all cases where the corresponding result is false, there are nontrivial implications for the group of birational automorphisms
$\Bir(X)$ \cite[\S 4]{BirMot}.
On the other hand, when the result is true, one obtains control over a truncated
Grothendieck group of varieties 
\cite[\S 3.2]{LSZ20},
\cite[\S 2.3]{BirMot}.
Thus Theorem~\ref{thm:main-intro} closes an important borderline case for this vanishing question. 
The new nontrivial cases that we check are when $X$ is a conic bundle over a non-ruled surface and when $X$ is a rational surface fibration over a positive genus curve. The proofs rely on Vanishings I, II and III. 

\subsection{Applications}

Let us now explain  some applications of horizontal and vertical invariants to the nonvanishing of motivic invariants.
The exceptional divisors of the Hassett--Lai Cremona transformations 
$\phi \in \Bir(\P^4_\k)$~\cite{HassettLai}, which were used in \cite{BirMot} when $\chr(\k) = 0$
to prove that $c(\Bir(\P^4_\k)) \ne 0$, form a
bounded family, since the corresponding blowup centers are models of K3 surfaces of degree $12$.
It is therefore natural to ask whether
the image $c(\Bir(\P^4_\k))$ is generated by a bounded family (see Definition~\ref{def-unbounded}). Explicitly, this would mean that there are only finitely many types of centers one can blow up that contribute to the nonvanishing of $c(\phi)$, while most types of centers cancel out.

We show that this is not the case: $c(\Bir(\P^n_\k))$ for $n \ge 4$ is unbounded in a very strong sense. 
To formulate the result we need to introduce a filtered group homomorphism
$\Burn_*(\k) \to \Burn_*(\k)$ which sends
a class $[X]$ to the class $[B]$ of the base of the MRC fibration $X \overset{\pi}\dto B$ (see~\S\ref{ss:MRC}).

\begin{thm}[see Theorem~\ref{thm-unbounded-c-B}]
\label{thm-unbounded} 
Let $\k$ be a field of characteristic zero.
Assume that $X$ is an $n$-dimensional variety birational to $B \times \P^3$
for some geometrically integral variety $B$ of positive dimension (for example $X = \P^n$ with $n \ge 4$).
Then the image
$$\Ima\left(\Bir(X) \overset{c}{\longrightarrow} \Burn_{n-1}(\k) \overset{\MRC}{\longrightarrow} \Burn_{\le n-1}(\k)\right)$$
contains
a geometrically unbounded subgroup of $\Burn_{n-2}(\k)$.
\end{thm}

In particular we have $c(\Bir(X)) \ne 0$ which strengthens the nonvanishing result of \cite[Theorem 4.4(b)]{BirMot}.
Informally speaking Theorem \ref{thm-unbounded} says that the image of $c(\Bir(X))$ is unbounded and that it always contains elements of the maximal MRC base dimension $n-2$ corresponding to codimension $2$ blow up centers.
The unboundedness aspect is related to the question in what sense birational self-maps of a given variety can be bounded. This question is explored in detail for threefolds in~\cite{BCDP}. In particular, by \cite[Theorem 1.1]{BCDP} many classes of rationally connected threefolds admit a sequence of birational automorphisms blowing up curves of unbounded genus. By contrast, Theorem \ref{thm-unbounded} shows that a similar phenomenon occurs in higher dimension, even after canceling centers with birational MRC bases.

Our proof of Theorem \ref{thm-unbounded}
builds 
on the threefold nonvanishing examples in~\cite{BirMot}, where the centers are curves of genus $1$ defined over $\k(B)$. As one of the steps in the proof, we show in 
Corollary~\ref{cor:P3-kB} that motivic invariants are always nontrivial for $\Bir(\P^3_{\k(B)})$,
which  extends \cite[Theorem 1.2(1)]{BirMot} 
where the same result was proven when $\k$ is a number field, algebraically closed field or finite field.
To prove  Theorem \ref{thm-unbounded} we spread  out those curves of genus $1$ to elliptic fibrations  $Y \to B$ with Kodaria dimension $\kappa(Y) = \dim(B)$.
Controlling the resulting motivic invariant for elements of $\Bir(\P^3 \times B)$
requires a careful analysis of the induced elliptic fibrations.

Finally, under the same assumptions as in Theorem \ref{thm-unbounded},
we obtain the following two corollaries from the maximality of the MRC base dimension claim.

\begin{corollary}[see Corollary \ref{cor:abelianizations}]
For any $k \ge 3$, 
the canonical morphism between abelianizations
\[
\Bir(\P^k \times B)^\ab \to
\Bir(\P^{k+1} \times B)^\ab 
\]
is not surjective.
\end{corollary}

This strengthens
\cite[\S2]{Szymik}, where this result was proved for $\k = \C$ and $B = \Spec(\C)$, based on the homomorphisms constructed using Sarkisov link decomposition in \cite{BLZ} and \cite{BSY}.

\begin{cor}[see Corollary \ref{cor:gener-CB-SB}]
The group $\Bir(B \times \P^3)$ is not generated by pseudo-regularizable maps together with birational maps preserving a conic bundle or a rational surface fibration.    
\end{cor}

This result is new already in the case when $B = \P^k$, $k > 0$, 
where it reproves and 
 strengthens the known results that $\Cr_n(\k)$ 
 with $n \ge 4$ is not generated
by linear automorphisms and de Jonqui\`eres maps~\cite[Theorem C]{BLZ}
(which also holds for $n = 3$),
or by pseudo-regularizable elements \cite[Theorem 1.2]{BirMot}, \cite[Theorem 1.2]{GLU} (which is currently unknown for $n = 3$).
Our approach to proving such results is entirely different from Blanc--Lamy--Zimmermann~\cite{BLZ} as we rely on motivic invariants while the proof of \cite{BLZ} is using Sarkisov link decomposition.

\medskip

\subsection{Conjectural description of the image of $c$}

We would like to finish the Introduction with  some speculations regarding the image $c(\Bir(X))$, which is currently unknown whenever it is nontrivial.
In the simplest nontrivial case $X = \P^4_\C$,
known elements that appear in $c(\Bir(\P^4_\C))$ are generated by the differences 
\begin{equation}\label{eq:P1-S-Sp}
[\P^1]([S] - [S']), \text{ where $S$ and $S'$ are D-equivalent K-nef surfaces. }
\end{equation}
Namely, $S$ and $S'$ can be D-equivalent K3 
surfaces \cite{BirMot} obtained from the Hassett--Lai map \cite{HassettLai}
or D-equivalent elliptic surfaces of Kodaira dimension $\kappa = 1$ as constructed in the proof of Theorem \ref{thm-unbounded} in this paper 
(the D-equivalence for such pairs of surfaces is a result of Bridgeland \cite{Bridgeland-elliptic}).
We propose the following:
\begin{conjecture}\label{conj:D-eq}
The image $c(\Bir(\P^4_\C))$ is generated
by elements of the form \eqref{eq:P1-S-Sp}.
\end{conjecture}

In particular, surfaces of Kodaira dimension $\kappa = 2$ conjecturally will
not contribute to the image $c(\Bir(\P^4_\C))$,
because for such surfaces D-equivalence implies birationality~\cite[Theorem 2.3]{Kawamata-DK}.

Conjecture \ref{eq:P1-S-Sp} would follow if the theory of Hodge atoms constructed by Katzarkov, Kontsevich, Pantev and Yu \cite{KKPY-atoms} admits a lifting to derived categories, specifically
if  derived categories of rational
$4$-dimensional smooth projective varieties  admit canonical, up to mutations, semiorthogonal decompositions which are compatible with smooth blow ups.
This has been conjectured by Kontsevich as well as by Halpern-Leistner \cite{DHL-MMP}, and is currently known in dimension up to two \cite{ESS-atoms}.

Now consider $X = \P^3_\k$ for a field $\k$.
Recall that for an algebraically closed field $\k$ of characteristic zero, we have
$c(\Bir(\P^3_\k)) = 0$ by
Theorem \ref{thm:main-intro}.
Motivated by Corollary \ref{cor:P3-kB}, our previous work \cite{BirMot}, 
known examples of L-equivalence~\cite{KuznetsovShinder, ShinderZhang}, 
and the theory of Hodge atoms~\cite{KKPY-atoms},
we propose a conjectural description for $c(\Bir(\P^3_\k))$ over nonclosed fields.
The curves $C$ and $C'$ from Corollary \ref{cor:P3-kB}
are the D-equivalent and L-equivalent curves from \cite{ShinderZhang},
and 
all currently constructed nontrivial elements in 
$c(\Bir(\P^3_\k))$ 
are generated by such differences.
Furthermore, only curves with geometric irreducible components of genus $g \le 1$
can contribute \cite[Proposition 2.6]{BirMot}, and
there is no nontrivial D-equivalence among
curves with geometric irreducible components of genus
$g = 0$ by~\cite{BO-reconstruction}, because they are Fano.
Therefore we propose the following:

\begin{conjecture}
\label{conj:D-eq-P3}    
For any field $\k$, the image $c(\Bir(\P^3_\k))$ is generated by 
the differences of the form
$\P^1 \cdot ([C] - [C'])$
where $C$ and $C'$ are  smooth projective curves that are
D-equivalent and
whose geometric irreducible components have genus $1$.
\end{conjecture}

\ssec*{Notation and conventions}

All varieties are integral, separated and of finite type over a field $\k$, but not necessarily geometrically integral.

\section{Horizontal and vertical motivic invariants}

In this section we set up the machinery of horizontal and vertical motivic
invariants for birational maps between
dominant maps
$\pi_1\colon X_1 \dto B_1$ and
$\pi_2\colon X_2 \dto B_2$.
If we assume that $\pi_1$ and $\pi_2$ are regular,
the theory is easier to set up; see Lemma \ref{lem-chorvar},
which can be considered as a definition in this case.
In general however, we need to pass to varieties over fields
$\k(B_1)$ and $\k(B_2)$ respectively, with birational maps acting on these base fields. We describe this formalism and define the corresponding motivic invariants in \S\ref{ss:relative}.
One of the nontrivial inputs in this direction is the vanishing result for surfaces, Theorem \ref{thm:intrinsic-centers}.
In \S \ref{ss:Stein} we introduce the rational Stein factorization which allows us to reduce computations of horizontal and vertical motivic invariants to the case where the geometric generic fibers of $\pi_1$ and $\pi_2$ are irreducible.

\subsection{Motivic invariants in the relative setting}
\label{ss:relative}

By $\ul{\Bir}/\k$, we mean the groupoid whose objects are algebraic varieties over $\k$ and whose morphisms are $\k$-birational maps.
This groupoid is anti-equivalent to 
the groupoid whose objects are 
finitely generated
field extensions $\k \subset L$ and
whose morphisms are $\k$-isomorphisms of field extensions.
We also use  
the graded Burnside group $\Burn_*(\k)$ \cite{KontsevichTschinkel},
which is
freely generated by birational isomorphism classes of $\k$-varieties (note that \cite{KontsevichTschinkel}  assumes that $\chr(\k) = 0$, and only uses smooth varieties, but we allow birational classes of any varieties).

The group $\Burn_*(\k)$ admits a graded ring structure, 
$$[X] \cdot [Y] = \sum_{i=1}^m [F_i],$$ 
where $F_i$ are the irreducible components of $(X \times_\k Y)_\redu$.
Note that if $\chr(\k) = 0$, then $X \times_\k Y$ is reduced, and
if $\k$ is algebraically closed, then $X \times_\k Y$ is integral.

We need to consider the relative versions
of $\ul{\Bir}/\k$ and $\Burn_*(\k)$,
which we denote by
$\wt{\ul{\Bir}/\k}$ and $\BURN_{*,*}(\k)$. The groupoid of relative birational types $\wt{\ul{\Bir}/\k}$ 
is defined as follows. 
Objects of $\wt{\ul{\Bir}/\k}$ are morphisms 
$X \to \Spec(\F)$ for a finitely generated field extension 
$\F/\k$, which makes $X$ 
an irreducible variety over $\F$. 
We write $X/\F$ for such an object.
A morphism between $X_1/\F_1$ and $X_2/\F_2$
is a 
commutative
diagram
\[\xymatrix{
X_1 \ar[d]_{\pi_1} \ar@{-->}[rr]^\phi && X_2 \ar[d]^{\pi_2} \\
\Spec(\F_1) \ar[rr]^\sigma_\simeq \ar[dr] & & \Spec(\F_2) \ar[dl] \\
& \Spec(\k) & 
}\]
where $\gs$ is an isomorphism of $\k$-extensions and
$\phi$ is a birational map if we regard both $X_1$ and $X_2$ as varieties over $\F_1$ (or $\F_2$). 
We sometimes denote such a morphism by 
$$\phi\colon X_1/\F_1 \dto X_2/\F_2.$$
Note that the category $\wt{\ul{\Bir}/\k}$ is equivalent to
the opposite category of the category of finitely generated field extensions $K/\F$ for some finitely generated $\k$-extension $\F$, with $\k$-isomorphisms of field extensions as morphisms.

Since every finitely generated field extension of $\k$ is realized by the function field of an algebraic variety,
the groupoid $\wt{\ul{\Bir}/\k}$ is equivalent to the groupoid whose objects are rational dominant maps $\pi\colon X \dto B$ between $\k$-varieties, and
whose morphisms are square birational maps as in~\eqref{eqn-introsq}.

For $n, d \ge 0$, we define the \emph{big Burnside group}
$\BURN_{n,d}(\k)$ as the free abelian group generated
by the isomorphism classes $[X/\F]$ of objects in $\wt{\Bir/\k}$ with $\dim(X/\F) = n$ and $\trdeg(F/\k) = d$. 
We set $\BURN_{*,*}(\k) = \bigoplus_{n,d} \BURN_{n,d}(\k)$, 
so that
the map
\[
\BURN_{*,*}(\k) \to \Burn_*(\k)
\]
sending $[X/\F]$ to 
$[\F(X)/\k]$ 
is a homomorphism of
graded abelian groups, with respect to the total degree on $\BURN_{*,*}(\k)$.

In this setting, we slightly generalize the motivic invariant $c(\phi)$ \cite{BirMot, LSZ20} defined by \eqref{eqn:c-intro}.
For every $\phi\colon X_1/\F_1 \dto X_2/\F_2$ as above, 
with $\dim(X_1) = \dim(X_2) = n$ and $\trdeg(F_1/\k) = \trdeg(F_2/\k) = d$,
we define 
\begin{equation}\label{eqn-genmotinv}
    c(\phi)
\cnec 
\sum_{E \in \ExDiv(\phi^{-1})}
 [E/\F_2]
 \ \ -  
 \sum_{D \in \ExDiv(\phi)} [D/\F_1] \ \ \in \ \ \BURN_{n-1,d}(\k)
\end{equation}
Here $\ExDiv(-)$ denotes the set of exceptional divisors of a map \cite[\S2.1]{BirMot}.
When $\F_1 = \F_2 = \k$ and $\sigma = \id$,
this definition of $c(\phi)$ coincides with that 
in~\cite{BirMot}.

\begin{lem}\label{lem:new-c-additive}
Given $\phi\colon X_1/\F_1 \dto X_2/\F_2$
and $\psi\colon X_2/\F_2 \dto X_3/\F_3$
we have 
\begin{equation}\label{eq:c-field-add}
c(\psi \phi) = c(\phi) + c(\psi).
\end{equation}
\end{lem}
\begin{proof}
We have a commutative diagram
\[\xymatrix{
X_1 \ar[d]^{\pi_1} \ar@{-->}[r]^\phi & X_2 \ar@{-->}[r]^\psi \ar[d]^{\pi_2} & X_3\ar[d]^{\pi_3} \\
\Spec(\F_1) \ar[r]^\sigma_\simeq & \Spec(\F_2)  \ar[r]^\tau_\simeq & \Spec(\F_3) \\
}\]
We consider $X_1$, $X_2$ and $X_3$ as varieties over the same field $\F_3$, via the morphisms $\tau\sigma\pi_1$, $\tau\pi_2$ and $\pi_3$
so that the birational maps $\phi$ and $\psi$ are also over $\F_3$.
Then $c(\phi)$, $c(\psi)$, $c(\phi\psi)$ coincide with those defined in \cite{BirMot,LSZ20}
and we can use the additivity from \cite[Lemma 2.2]{BirMot}. 
\end{proof}

For a variety $X/\F$ 
and an isomorphism $\sigma\colon \F \xrightarrow{\sim}
\F'$, 
 we write
$\sigma X$ for the composition 
$$X \to \Spec(\F) \overset{\sim}{\to} \Spec(\F'),$$ 
and similarly for morphisms and birational maps 
between $\F$-varieties. Another way to think about $\sigma X$ is to notice that there is an $\F'$-isomorphism $\sigma X \simeq X \times_\F \F'$. 
For every birational map $\phi\colon X \dto Y$ between $\F$-varieties, 
it is clear that
$$c(X \overset{\phi} \dto Y) = c(\gs X \overset{\gs \phi} \dto \gs Y).$$

Since $\sigma X$ and $X$ are isomorphic as schemes, many standard numerical properties, such as the Kodaira dimension, 
the Picard rank and the (anti)canonical degree, are preserved under this operation. 
However, in general if $\sigma \in \Aut(\F)$, then $\sigma X$ and $X$ are not isomorphic as $\F$-varieties. 

The following is a variant of the main result of~\cite{LSZ20}
in the setting of relative birational types.
We will rely on this theorem when we analyze horizontal invariants between maps of relative dimension two, see Proposition \ref{prop:vanishingI}.

\begin{theorem}\cite{LSZ20}\label{thm:intrinsic-centers}
Let $\k$ be a field of characteristic zero.
There exists a canonical assignment
\[
S/\F \mapsto \cM(S/\F) \in \BURN_{0,\trdeg(\F/\k)}(\k),
\]
for geometrically integral 
surfaces over finitely generated extensions $\F/\k$,
satisfying the following properties:
\begin{enumerate}
    \item[(a)] For any $\phi\colon S_1/\F_1 \dto S_2/\F_2$
    we have $c(\phi) = \cM(S_2/\F_2) \cdot [\P^1_{\F_2}] - \cM(S_1/\F_1) \cdot [\P^1_{\F_1}]$. 
    \item[(b)] For every isomorphism $\sigma \colon \F \xrightarrow{\sim} \F'$,
    $\cM(\sigma S/\F') = \cM(S/\F)$
    
\end{enumerate}

In particular for any $S/\F$ and any birational map $\phi\colon S/\F \dto S/\F$  which possibly acts nontrivially on the base field $\F$ we have $c(\phi) = 0$.
\end{theorem}

\begin{proof}
Our proof is the same as in \cite[\S5]{LSZ20}.
We first define $\cM(S)$
for minimal smooth projective geometrically integral surfaces as follows.

\begin{enumerate}
    \item If $S$ is geometrically irrational or 
    $K_S^2 \le 4$,
    set $\cM(S/\F) \cnec \rk(\NS(S)) \cdot  [\Spec (\F)]$; 
    \item If $S$ is geometrically rational with degree $K_S^2 \ge 5$,
    define $\cM(S/\F)$ using Hilbert schemes of curves on $S$ of certain anticanonical degrees as in \cite[Definition 5.2]{LSZ20}.
\end{enumerate}

We check that with this definition conditions (a) and (b) hold for minimal smooth projective geometrically integral surfaces.
Condition (b) is trivially satisfied in case (1),
and in case (2)
it follows from
$\sigma\cH_j(S) = \cH_j(\sigma S)$,
where $\cH_j(S)$ is the Hilbert schemes parameterizing curves of degree $j$ (with respect to $-K_S$).
Using condition (b), to check (a) we can replace $S_1/\F_1$ by $\sigma S_1/\F_2$ and assume that $\phi$ induces the identity map on the base field.
Under this assumption any birational map 
between minimal surfaces $S$, $S'$
satisfies condition (a) 
by~\cite[Proposition 4.4, Proposition 5.5]{LSZ20}.

The paper \cite{LSZ20} uses a different normalization for $c(\phi)$, as a zero-dimensional class, not as a combination of exceptional divisors as we do in \cite{BirMot}
and in this paper. Let us write $c^{o}(\phi)$ for the invariant from \cite{LSZ20} so that by construction we have, for every birational map
$\phi \colon S_1 \dto S_2$  acting trivially on the base field $\F$,
\[
c(\phi) = c^{o}(\phi)[\P^1].
\]

Now, for an arbitrary geometrically integral surface $S$ we take a minimal smooth projective model $\phi\colon S \dto \ol{S}$
and define
\begin{equation}\label{eq:M-S-arbitrary}
\cM(S) := \cM(\ol{S}) - c^o(\phi).
\end{equation}
This is independent of the choice of $\phi$
because if $\psi\colon S \dto \ol{S}'$
is another such model, then
using (a) applied to $\phi \psi^{-1}$
and the additivity of $c$ we get
\[
\cM(\ol{S}) - c^o(\phi) = 
\cM(\ol{S}') + c^o(\phi \psi^{-1}) - c^o(\phi) = 
\cM(\ol{S}') - c^o(\psi). 
\]

A very similar computation extends
(a) 
from birational maps between minimal surfaces
to birational maps between arbitrary surfaces.
For (b) we note that if $\phi\colon S \dto \ol{S}$ is a minimal smooth projective model of $S$, then $\sigma \phi$ can be taken as a model of $\sigma S$, so that using (b) for the minimal surface $\ol{S}$ we get
\[
\cM(\sigma S/\F') = 
\cM(\sigma\ol{S}/\F') - c^o(\sigma\phi) =
\cM(\ol{S}/\F) - c^o(\phi) = 
\cM(S/\F).
\]
\end{proof}

\ssec{Horizontal and vertical motivic invariants}
\label{ssec:hor-vert-c}

Let
$\pi_1 \colon X_1 \dto B_1$ and $\pi_2\colon X_2 \dto B_2$
be two dominant rational maps of $\k$-varieties.
We say that $\phi\colon X_1 \dto X_2$
is a birational map between $\pi_1$
and $\pi_2$
if it fits into a commutative diagram
\begin{equation}\label{eqn-squareRat}
    \xymatrix{
    X_1 \ar@{-->}[d]_{\pi_1} \ar@{-->}[r]^\phi & X_2 \ar@{-->}[d]^{\pi_2}\\
    B_1 \ar@{-->}[r]^\sigma
    & B_2
    }
\end{equation}
for a birational map $\sigma$. Note that $\sigma$ is uniquely determined by $\phi$.
We write 
$\Bir(\pi_1,\pi_2)$
for the set of birational maps between $\pi_1$ and $\pi_2$.
We also write $\Bir(\pi)$ for $\Bir(\pi_1, \pi_2)$ when $\pi_1 = \pi_2 = \pi \colon X \dto B$. 
We use the notation
$$\Bir(X/B) \le \Bir(\pi)$$
for the subgroup of $\Bir(\pi)$
consisting of birational automorphisms of $X$ which descend
to the identity on $B$.

For $i = 1,2$,
    let $U_i \subset X_i$ be the open complement of the indeterminacy locus of $\pi_i$. 
Then passing to the generic points of the bases
we obtain a bijection
\begin{equation}\label{eq:bijection-Mor-generic}
\begin{split}
\Bir(\pi_1,\pi_2) & \simeq 
\Mor_{\wt{\Bir/\k}}(U_{1, \k(B_1)}/\k(B_1),U_{2, \k(B_2)}/\k(B_2)) \\
\phi & \mapsto \phi^\eta.    
\end{split}
\end{equation}
The varieties $U_{1, \k(B_1)}$ and $U_{1, \k(B_2)}$ 
have the same dimension $n$, 
equal to the relative dimension of $\pi_1$ and $\pi_2$.
Let $d = \dim(B_1) = \dim(B_2)$.

\begin{Def}\label{def-chor1}
The \emph{horizontal motivic invariant}
$c_\hor(\phi)$
is the image of $c(\phi^\eta)$
defined in \eqref{eqn-genmotinv}
under the forgetful map $\BURN_{n-1,d}(\k) \to \Burn_{n+d-1}(\k)$.
The \emph{vertical motivic invariant} is defined by
\[
c_{\ver}(\phi) \cnec c(\phi) - c_{\hor}(\phi) \in \Burn_{n+d-1}(\k).
\]
\end{Def}

\begin{lemma}\label{lem:add}
Both $c_{\hor}$ and $c_{\ver}$ are additive under compositions: 
if $\phi \in \Bir(\pi_1, \pi_2)$
and $\psi \in \Bir(\pi_2, \pi_3)$ then
\[
c_\hor(\psi \circ \phi) = c_\hor(\phi) + c_\hor(\psi),
\quad
c_\ver(\psi \circ \phi) = c_\ver(\phi) + c_\ver(\psi).
\]
\end{lemma}

\begin{proof}
Additivity of $c_{\hor}$ follows from Lemma \ref{lem:new-c-additive}
and the additivity of the forgetful map
$\BURN_{*,*}(\k) \to \Burn_*(\k)$.
Since $c = c_\hor + c_\ver$
is also additive by Lemma~\ref{lem:new-c-additive},
$c_\ver$ is additive as well.
\end{proof}

Additivity of the invariants has the following  useful consequence. 
Let $\pi \colon X \dto B$ and $\phi \in \Bir(\pi)$. Consider arbitrary dense open subsets $U \subset X$, $V \subset B$. Write $\pi^{U}_V\colon U \dto V$ and $\phi^U_V \colon U \dto U$ for the induced dominant rational maps obtained by restriction of $\pi$ and $\phi$ respectively, 
so that $\phi^U_V \in \Bir(\pi^U_V)$.
The open embedding $j_U \colon U \hookrightarrow X$ can be considered as an element $j_U \in \Bir(\pi^U_V, \pi)$, and we have 
$\phi^U_V = j_U^{-1} \circ \phi \circ j_U$. Thus Lemma \ref{lem:add} implies that
\begin{equation}\label{eq:phi-phiU}
c_\hor(\phi) = c_\hor(\phi^U_V), \quad    
c_\ver(\phi) = c_\ver(\phi^U_V).
\end{equation}

We will often work with regular morphisms $\pi_1$ and $\pi_2$, 
in which case $c_\hor(\phi)$ and $c_\ver(\phi)$ 
have a clear geometric meaning
as described in the following lemma.

\begin{lem}\label{lem-chorvar}
If $\pi_1$ and $\pi_2$ are regular in codimension $1$ (i.e. defined on the generic points of all divisors), then for every $\phi \in \Bir(\pi_1,\pi_2)$, 
    we have
\begin{equation}\label{eqn-chorvar}
c_{\hor}(\phi) = 
\sum_{\substack{E \in \ExDiv(\phi^{-1}) 
\\ \ol{\pi_2(E)} = B_2}} [E]
 \ \ -  \sum_{\substack{D \in \ExDiv(\phi)\\ 
\ol{\pi_1(D)} = B_1}} [D] \ \ \in \ \ \Burn_*(\k),
\end{equation}
\begin{equation}\label{eqn-cvervar}
c_{\ver}(\phi) = 
\sum_{\substack{E \in \ExDiv(\phi^{-1}) 
\\ {\ol{\pi_2(E)}} \ne B_2}} [E]
 \ \ -  \sum_{\substack{D \in \ExDiv(\phi)\\ 
{\ol{\pi_1(D)}} \ne B_1}} [D] \ \ \in \ \ \Burn_*(\k).
\end{equation}
\end{lem}

\begin{proof}
The first formula is a consequence of the one-to-one correspondence
\begin{equation*}
  \begin{split}
      \Set{\text{Prime divisors on } X_i \text{ dominating } B_i} & \xto{1:1} \Set{ \text{Prime divisors on } X_{i,\k(B_i)}} \\
      D & \mapsto D|_{\k(B_i)} 
  \end{split}  
\end{equation*}
with inverse defined by taking Zariski closure.
The second formula follows from the first one and the definition of $c_\ver(\phi)$.
\end{proof}

\ssec{Rational Stein factorizations}
\label{ss:Stein}

The following is a rational version of the usual Stein factorization  \cite[Corollary III.11.5]{Hart},
a notion that already appears in~\cite[Lemma 4.7]{DSW-Weak-Approx}.

\begin{Def}\label{def-RStein}
Let $\pi\colon X \dto B$ 
be a rational dominant map of $\k$-varieties.
The \emph{rational Stein factorization} of $\pi$ is the factorization
$$\pi \colon X \overset{\wt{\pi}} \dto \wt{B} \overset{f}\to B$$
where $\wt{B}$ is the normalization of $B$ in $\k(X)$.
We refer to $\wt{\pi}$ and $f$ as the \emph{connected part} and the \emph{finite part} of $\pi$ respectively.
The degree of $\k(\wt{B})/\k(B)$ 
in the rational Stein factorization
is called the \emph{Stein degree} of $\pi \colon X \dto B$.
\end{Def}

\begin{remark} 
The notion of Stein degree for the usual Stein factorization of a regular proper morphism was introduced in \cite[\S3]{BirkarModAlgVar}, where it was conjectured that
the horizontal components of the boundary for log Calabi--Yau fibrations have bounded Stein degree over the base.
This conjecture was recently proven in \cite{BirkarQu1};
see also \cite{BirkarQu2}.
Note that by Proposition \ref{prop:Stein}(2), if $X$ is normal and $\pi$ is surjective, the Stein degree as we define it coincides with the one in \cite{BirkarModAlgVar}.
\end{remark}

\begin{prop} \label{prop:Stein}
Let $\pi\colon X \dto B$ is a rational dominant map.
The rational Stein factorization in Definition~\ref{def-RStein}
has the following properties:
\begin{enumerate}
    \item $\wt{B}$ is normal and $f$ is a finite morphism.
    \item 
    If $\pi\colon X \to B$ is a proper surjective morphism, with (usual) Stein factorization 
    \[
    X \xrightarrow{\wt{\pi}} \wt{B} \xrightarrow{f} B,
    \]
    then the rational Stein factorization of $\pi$ is 
    $$X \overset{\nu^{-1}\wt{\pi}}\dto \wt{B}^\nu \overset{f\nu}\to B$$ 
    where $\nu\colon \wt{B}^{\nu} \to \wt{B}$ is the normalization of $\wt{B}$. In particular, if we additionally assume that $X$ is normal, then the rational Stein factorization of $\pi$ coincides with the usual one.
    \item 
    Assume that $\pi$ is a proper morphism 
    from a normal variety $X$.
    The geometric generic fiber of $\wt{\pi}$ is irreducible. It is integral if $\chr (\k) = 0$.
    \item Every birational map $\phi \in \Bir(\pi)$ 
induces a birational map on $\wt{B}$, so that we have 
$$\Bir(\pi) \subset \Bir(\wt{\pi}) \subset \Bir(X).$$
\end{enumerate}
\end{prop}

\begin{proof} 
    (1) and (4) follow from the construction.
    
For (2), let $X \dto B' \to B$ denote the rational Stein factorization of $\pi$.
By definition, $\wt{B}$ is the normalization of $B$ in $X$, thus
by the functoriality of relative normalizations~\cite[035J]{stacks-project}, we have a commutative diagram 
$$
    \xymatrix{
    \Spec (\k(X)) \ar[d] \ar[r] & X \ar[d]^{\ti{\pi}} \\
    B' \ar[r]   & \wt{B}
    }
$$
over $B$. Note that by (1) and the properness of $\pi$,
both $B'$ and $\wt{B}$ are finite over $B$, so $B' \to \wt{B}$ is finite as well. 
Since $\wt{B}$ is normal in $X$, necessarily $B' \to \wt{B}$
has degree $1$. As $B'$ is normal by (1), it follows from the Zariski main theorem that $B' \to \wt{B}$ is the normalization of $\wt{B}$. Finally, if $X$ is further assumed normal, then $\wt{B}$ is already normal~\cite[Lemma 035L]{stacks-project}. This proves (2).

    For (3), we note that since $X$ is an integral $\k$-variety,
    the generic fiber $X_{\k(\wt{B})}$ of $\wt{\pi}$ is irreducible and reduced~\cite[Lemma 054Z]{stacks-project}.
    As $\k(\wt{B})$ is separably closed in $\k(X)$
    by construction, 
    the geometric generic fiber of $\wt{\pi}$ is irreducible
    by~\cite[Corollary 3.2.14.(d)]{Liu-AG}.
    The second statement follows from~\cite[Lemma 2.6.4]{MR2445111}
    and~\cite[Corollary 3.2.14.(c)]{Liu-AG}.
\end{proof}

Using rational Stein factorizations, the computation of horizontal and vertical
motivic invariants reduces to the case of 
birational maps between rational dominant maps with irreducible
geometric generic fibers.

\begin{cor}\label{cor-SFchor}
    Let $\pi\colon X \dto B$ be a dominant
rational map of $\k$-varieties and let
$\ti{\pi} \colon X \dto \wt{B}$ 
be the connected part of the rational Stein factorization of $\pi$.
Then 
for any $\phi \in \Bir(\pi) \subset \Bir(\ti{\pi})$, the invariants $c_\hor(\phi)$
and $c_\ver(\phi)$
do not depend on whether $\phi$ is considered as an element of $\Bir(\pi)$ or of $\Bir(\wt{\pi})$.
\end{cor}

\begin{proof}
Let $\ti{\phi} \in \Bir(\wt{\pi})$
denote the element identified with $\phi \in \Bir(\pi)$ by Proposition \ref{prop:Stein}(4).
Let $U \subset X$ be a nonempty Zariski open
such that both $\pi|_U$
and $\ti{\pi}|_U$ are regular. 
Since $f\colon \wt{B} \to B$ is finite, a prime divisor 
    in $U$ is horizontal over $\wt{B}$ if and only if 
    it is horizontal over $B$.
By Lemma~\ref{lem-chorvar},
this implies the middle equality in
$$c_\hor(\phi) = c_\hor(\phi|_{U}) 
= c_\hor(\ti{\phi}|_{U}) = c_\hor(\ti{\phi}).$$
The first and the last equalities follow from \eqref{eq:phi-phiU}.
This proves the statement for $c_\hor(\phi)$,
and thus for $c_\ver(\phi) = c(\phi) - c_\hor(\phi)$.
\end{proof}

We will use the following lemma in Section~\ref{sec-ubd}.

\begin{lem}\label{lem-bdStein} 

    Let $\pi\colon \cX \to T$ and $\mu\colon \cB \to T$ be morphisms
    of $\k$-varieties
    such that 
    every fiber of $\pi$ and $\mu$ is  
    geometrically integral. 
    Let
        $$
     \xymatrix@C=1em{
    \cX \ar[dr]_\pi\ar@{-->}[rr]^f  & & \cB  \ar[dl]^\mu\\
     & T & }
    $$
    be a dominant rational map over $T$ 
    such that the indeterminacy locus of 
    $f$ does not contain any fiber of $\pi$.
    Then there exists a locally closed stratification 
    $T = \bigsqcup_i T_i$ such that
    the Stein degree of the fiber $f_t \colon \cX_t \dto \cB_t$ of $f$ over
    $t \in T_i(\ol{\k})$ is constant for each $i$. 
\end{lem}

\begin{proof}
It suffices to show that there exists a nonempty 
Zariski open $U \subset T$
such that the Stein degree of the fiber $f_t \colon X_t \dto B_t$ of $f$ over
    $t \in U(\ol{\k})$ is constant. 

Up to shrinking $T$, 
we can assume that the non-normal locus of $\cX$
does not contain any fiber of $\pi$.
We can find a regular and proper replacement 
$f' \colon \cX' \to \cB$ of $f$ with $\cX'$ normal by first resolving the indeterminacy (by taking the normalization of the graph) and then applying Nagata compactification over $\cB$.
By further shrinking $T$, we can still assume that
every fiber of $\pi$ is geometrically integral.
Note that by construction, the restriction of 
$\cX' \to \cX$ to every fiber $\cX'_t$ of $\cX' \to T$
is birational onto $\cX_t$ for every $t \in T(\ol{\k})$. 
As Stein degree is a birational invariant, 
we can replace $f$ by $f'$.

Let
$$\cX \xto{\wt{f}} \wt{\cB} \to \cB$$
be the Stein factorization of $f \colon \cX \to \cB$.
Up to shrinking $\cX$, $\cB$, and $T$, we can assume that
$\pi$, $\mu$, and $\ti{f}$ are flat and surjective, 
with $f$ (and thus $\wt{f}$) 
remaining proper.
By Grauert's base change, the map
$$\cO_{\wt{\cB}_t} \simeq (\wt{f}_*\cO_\cX)|_{\wt{\cB}_t} \to  
\wt{f}_*\cO_{\cX_t}$$
is an isomorphism.
Thus
$$\cX_t \xto{\wt{f}} \wt{\cB}_t \to \cB_t$$
is the Stein factorization of the restriction 
$\cX_t \to \cB_t$ of $f$ to $\cX_t$.

It follows from Proposition~\ref{prop:Stein}(2)
that the Stein degree of $f_t$
is the degree of the finite morphism $\wt{\cB}_t \to {\cB}_t$,
which is constant for every $t \in U(\ol{\k})$
in some nonempty Zariski open $U \subset T$.
\end{proof}

\section{Computing motivic invariants}

In this section we present several ways of computing motivic invariants, in particular we prove some vanishing results for motivic invariants of self-maps. These are
Vanishing I (Proposition \ref{prop:vanishingI}),
Vanishing II (Corollary \ref{cor:vanishingII})
and Vanishing III (Corollary \ref{cor:vanishingIII}).
Along the way, we establish useful 
formulas to compute motivic invariants: 
Proposition \ref{pro:cvercons} in the regular flat case and
Theorem \ref{thm:c-ver-bir-triv} for vertical invariants for a special kind of regular morphisms that we call birationally trivial in codimension one (Definition \ref{def:bir-trivial-c1}). 

In \S\ref{ss:MRC} we recall some properties of MRC fibrations and relate them to motivic invariants. 
The main results in this direction are Theorem 
\ref{thm-van3fold} 
and Proposition \ref{prop:RCfib}.

\ssec{Vanishing results}

For $i = 1,2$,
let $\pi_i\colon X_i \dto B_i$ be dominant
rational maps of
$\k$-varieties.

\begin{lemma}\label{lem:rel-1}
 Assume that
 $\pi_1$ and $\pi_2$
 are regular, proper, and generically 
 smooth of
 relative dimension at most one.
 We have $c_{\hor}(\phi) = 0$
 for any $\phi \in \Bir(\pi_1, \pi_2)$.
 \end{lemma}

 \begin{proof}
 It suffices to note that every birational map
 between points or smooth proper curves is an isomorphism. 
 \end{proof}

\begin{prop}[{Vanishing I}]
\label{prop:vanishingI} 
    Let $\pi\colon X \dto B$ be a dominant rational map between 
    $\k$-varieties. 
    Suppose that  
    $$\dim X - \dim B \le 2.$$
    If $\dim X - \dim B = 2$, assume in addition that $\k$ has characteristic zero.
    Then
    $$c_\hor(\phi) = 0 \text{
    for any $\phi \in \Bir(\pi)$}.$$
\end{prop}

\begin{proof}

As we did in the proof of Lemma~\ref{lem-bdStein},
we can find a regular and proper replacement $\pi'\colon X' \to B$ of $\pi$ with $X'$ normal.
If $\dim(X) - \dim(B) \le 1$, we use Lemma \ref{lem:rel-1}.

If $\dim(X) - \dim(B) = 2$ and $\chr(\k) = 0$,
then
by Corollary~\ref{cor-SFchor} and Proposition~\ref{prop:Stein}(3), we can assume that
the geometric generic fiber of $\pi'$ is integral so that
the result
follows from Theorem \ref{thm:intrinsic-centers}.
\end{proof}

Our next goal is to 
state a general result for computing vertical motivic invariants assuming that $\pi_1$, $\pi_2$
are 
regular and flat, see Proposition \ref{pro:cvercons}.
First let us explain a simple formula for the usual motivic invariant $c(\phi)$ in terms of valuations on function fields.
For a normal variety $X$, we denote by $X^{(1)}$
the set of prime Weil divisors on $X$.
Every $\xi \in X^{(1)}$ defines a discrete valuation on $\k(X)$, but not every discrete valuation is of this form. Discrete valuations on $\k(X)$
arising from a divisor on a normal birational model of $X$ are called \emph{algebraic} \cite[Remark 2.23]{KollarMori}
and they admit a simple intrinsic characterization given in \cite[Lemma 2.45]{KollarMori}.

For a birational map $\phi\colon X_1 \dto X_2$,
between normal varieties,
both $X_1^{(1)}$ and $X_2^{(1)}$
can be considered as subsets of algebraic discrete valuations on the function field $\k(X_1) \simeq \k(X_2)$, identified via $\phi$.
For every algebraic discrete valuation $\xi$,
let us denote by $\ol{\xi}^{X_i}$ its center in $X_i$, that is the closure of the image of the generic point of the corresponding divisor.
In these terms, 
we have
\[
\ExDiv(\phi) = X_1^{(1)} \setminus X_2^{(1)},
\quad
\ExDiv(\phi^{-1}) = X_2^{(1)} \setminus X_1^{(1)}.
\]
Furthermore, for all $\xi \in X_1^{(1)} \cap X_2^{(1)}$,
divisors $\ol{\xi}^{X_1}$ and $\ol{\xi}^{X_2}$ are birational.
Hence 
by definition
\begin{equation}\label{eq:c-valuations-def}
c(\phi\colon X_1 \dto X_2) = 
\sum_{\xi \in X_1^{(1)} \cup X_2^{(1)}}
\left([\ol{\xi}^{X_2}] - [\ol{\xi}^{X_1}] \right),
\end{equation}
where $[\ol{\xi}^{X_i}]$ is zero if 
$\codim(\ol{\xi}^{X_i}) > 1$.
This is a finite sum because only divisors from the union $\ExDiv(\phi) \cup \ExDiv(\phi^{-1})$ can make nontrivial contributions. 
The dependence of the right-hand side of~\eqref{eq:c-valuations-def} on $\phi$ is encoded in the union $X_1^{(1)} \cup X_2^{(1)}$.

We have the following generalization of~\eqref{eq:c-valuations-def}.

\begin{prop}\label{pro:cvercons}

Let $\k$ be any field.
Let
$\pi_1\colon X_1 \to B_1$ and $\pi_2\colon X_2\to B_2$
be flat morphisms
between normal varieties
and $\phi \in \Bir(\pi_1, \pi_2)$.
We regard $B_1^{(1)}$ and $B_2^{(1)}$
 as subsets of valuations in the function fields $\k(B_1) \simeq \k(B_2)$,
identified through $\phi$.
We have
\[
c_\ver(\phi) = \sum_{\zeta \in B_1^{(1)} \cup B_2^{(1)}}
\left([\pi_2^{-1}(\ol{\zeta}^{B_2})] - [\pi_1^{-1}(\ol{\zeta}^{B_1})] \right),
\]
which is a finite sum.
Here 
$[\pi_i^{-1}(\ol{\zeta}^{B_i})]$ denotes the sum of the prime Weil divisors on $X_i$ in $\pi_i^{-1}(\ol{\zeta}^{B_i})$.
\end{prop}

As the proof shows, instead of flatness it suffices to assume in Proposition~\ref{pro:cvercons} that $\pi_1$ and $\pi_2$ do not map divisors to subsets of codimension at least two. 

\begin{proof}
By Lemma~\ref{lem-chorvar},
the vertical invariant is equal to the sum analogous to \eqref{eq:c-valuations-def}:
\begin{equation}\label{eq:c-ver-valuations}
c_\ver(\phi \colon X_1 \dto X_2) = 
\sum_{\xi \in X_{1,\ver}^{(1)} \cup X_{2,\ver}^{(1)}}
\left([\ol{\xi}^{X_2}] - [\ol{\xi}^{X_1}] \right),
\end{equation}
where 
$$X_{i,\ver}^{(1)} \cnec \Set{\xi \in X_i^{(1)} | \ol{\pi_i(\ol{\xi}^{X_i})} \ne B_i}.$$
As $\pi_i$ are flat, the closure of the image of a vertical prime divisor must be a prime divisor, so we have well-defined maps
$X_{i,\ver}^{(1)} \to B_i^{(1)}$ which agree on the intersections and define a map
\begin{equation}
X_{1,\ver}^{(1)} \cup X_{2,\ver}^{(1)} 
\xrightarrow{\pi^{(1)}} B_1^{(1)} \cup B_2^{(2)}.
\end{equation}
Splitting the sum \eqref{eq:c-ver-valuations}
over the fibers of $\pi^{(1)}$ 
gives
\begin{equation*}
c_\ver(\phi) = 
\sum_{\zeta \in B_1^{(1)} \cup B_2^{(1)}} \sum_{\pi^{(1)}(\xi) = \zeta}
\left([\ol{\xi}^{X_2}] - [\ol{\xi}^{X_1}] \right),
\end{equation*}
which
implies the result.
 \end{proof}

\begin{cor}[{Vanishing II}]
\label{cor:vanishingII}
For every  rational dominant map 
$\pi\colon X \dto B$, we have
$$c_\ver(\phi) = 0
\text{ for all
$\phi \in \Bir(X/B)$.}
$$
\end{cor}
\begin{proof} 

First assume that $\pi$ is a regular flat morphism between normal varieties.
The vanishing of $c_\ver(\phi)$ in this case
follows immediately from Proposition \ref{pro:cvercons}, because all the terms in the sum are zero.

In general, by generic flatness, there exists a
dense Zariski open $U \subset X$ such that
$\pi$ restricts to a regular flat morphism onto its image $V \subset B$.
Furthermore we can assume that $U$ and 
$V$ are normal. 
By~\eqref{eq:phi-phiU} we can replace $\pi$ and $\phi$ by $\pi^U_V$ and $\phi^U_V$ respectively and reduce to the special case explained above.
\end{proof}

Now we concentrate on  a special kind of morphisms.
 
\begin{Def}\label{def:bir-trivial-c1}
    A flat morphism $\pi\colon X \to B$ between normal varieties is called \emph{birationally trivial in codimension $1$}
    with fiber $F/\k$ if 
    for all but finitely many prime divisors $D$
    of $B$, 
     $\pi^{-1}(D)$ is irreducible and birational to $D \times F$
    over $D$.
\end{Def}

Note that since $\pi\colon X \to B$ is assumed to be flat, for every $\zeta \in B^{(1)}$ with $D = \ol{\zeta} \subset B$ and $X_\zeta = \pi^{-1}(\zeta)$ we have
$\ol{X_\zeta} = \pi^{-1}(D)$ and the condition of birational triviality in codimension $1$ can be equivalently restated as birationality between $\k(\zeta)$-varieties $X_{\zeta}$ and $F \times_{\k} {\k(\zeta)}$ for all but finitely many $\zeta$.

If $\k$ is algebraically closed, then $F$ must be birational to fibers of $\pi$ over general closed points of $B$.
We have the following examples for this notion;
we assume that $X$ and $B$ are normal.

\begin{example}\label{ex:gener-trivial}
If $X$ is birational to $F \times B$ over $B$, for example if $X$ is a Zariski locally trivial fiber bundle, then $\pi$ is birationally trivial in codimension $1$.    
This also applies if the generic fiber $X_{\k(B)}$ is rational over $\k(B)$.
\end{example}

\begin{example}\label{ex:bir-triv-curve} If $B$ is a curve over an algebraically closed field $\k$
and $\pi\colon X \to B$ is a flat morphism with rational general fibers, then 
$\pi$ is birationally trivial in codimension $1$. Note that this is not a particular case
of Example \ref{ex:gener-trivial}
since $X_{\k(B)}$ can still be irrational.
\end{example}

\begin{example}\label{ex:bir-triv-surf}
If $B$ is a surface over an algebraically closed field $\k$
and $\pi\colon X \to B$ 
is a flat morphism whose
generic fiber of  is a Severi--Brauer variety of dimension $n$, then $\pi$ is birationally trivial in codimension $1$ with fiber $\P^n$.
Indeed, there exists a Zariski open subset $U \subset B$
such that the restriction $\pi^{-1}(U) \to U$
is a Severi--Brauer fibration.
Thus restricting $\pi$ to every integral curve $D \subset U$ is a Severi--Brauer fibration which 
is
trivial over $\k(D)$ by Tsen's theorem.
\end{example}

\begin{example} The concept of birational triviality in codimension one depends on the base field.
For example,
$\{x^2 +y^2 = t\} \subset \A^2_{x,y} \times \A^1_t 
\to \A^1_t$ is 
birationally trivial in codimension $1$ over $\C$,
but not over $\R$.    
\end{example}

If $\pi\colon X \to B$ is birationally trivial in codimension one, we define its \emph{vertical divisorial defect} by the formula
\[
d(\pi) = \sum_{D \in B^{(1)}} [\pi^{-1}(D)] - [F \times D] \in \Burn_*(\k).
\]
Here, $[\pi^{-1}(D)]$ is the sum of 
the classes of prime Weil divisors in $\pi^{-1}(D)$ (without multiplicities).
Note that this sum is finite by our assumption on $\pi$. 

\begin{theorem}\label{thm:c-ver-bir-triv}
Let
$\pi_1$ and $\pi_2$
be birationally trivial morphisms in codimension $1$, with the same fiber $F$.
Then for all $\phi \in \Bir(\pi_1, \pi_2)$,
\begin{equation}
\label{eq:c-gentriv}
c_\ver(\phi) = 
d(\pi_2) - d(\pi_1)
+  c(\sigma) \cdot [F]. 
\end{equation}
\end{theorem}

\begin{proof}
    We can find a factorization of $\phi \in \Bir(\pi_1, \pi_2)$
    of the form
    $$
    \xymatrix{
X_1 \ar[d]_{\pi_1} & \wt{U_1} \ar@{_{(}->}[l]_{j_1} \ar@{-->}[r]^{\psi} \ar[d] & \wt{U_2} \ar[d]\ar@{^{(}->}[r]^{j_2}   & X_2 \ar[d]^{\pi_2} \\
B_1 & U_1 \ar@{_{(}->}[l] \ar[r]^{\sim} & U_2 \ar@{^{(}->}[r] & B_2 \\
}
    $$
 such that
\begin{enumerate}
    \item for $i = 1,2$ and every prime divisor 
    $D_i \subset B_i$ such that $D_i \cap U_i \ne \emptyset$,
    the preimage $\pi_i^{-1}(D_i)$ 
    is birational to $D_i \times F$;
    \item the outer squares are cartesian.
\end{enumerate}
By construction of the vertical invariant (see Lemma \ref{lem-chorvar}), we obtain using (1):
$$c_\ver(j_i) 
= \sum_{D_i \in B_i^{(1)} \bss U_i^{(1)}} [\pi_i^{-1}(D_i)]
= d(\pi_i) + \sum_{D_i \in B_i^{(1)} \bss U_i^{(1)}} [D_i] \cdot [F].
$$
Since $\wt{U_i} \to U_i$ are flat, Proposition~\ref{pro:cvercons}
and (1) imply that $c_\ver(\psi) = 0$.
Hence
$$c_\ver(\phi) = - c_\ver(j_1) + c_\ver(\psi) + c_\ver(j_2)
= d(\pi_2) - d(\pi_1) + c(\sigma) \cdot [F].$$
\end{proof}

We record the following
immediate consequence of Theorem~\ref{thm:c-ver-bir-triv}.

\begin{cor}[{Vanishing III}] \label{cor:vanishingIII}
If $\pi\colon X \to B$ is birationally trivial in codimension one 
and 
$c(\Bir(B)) = 0$,
then $c_\ver(\phi) = 0$
for all $\phi \in \Bir(\pi)$.
\end{cor}

The condition $c(\Bir(B)) = 0$ always holds when $\dim(B) = 1$
and when $\dim(B) = 2$ if $\k$
is a perfect field by the main result of
\cite{LSZ20}.

\subsection{MRC fibrations}
\label{ss:MRC}

Let us assume that $\k$ is of characteristic zero.
We will say that a variety $X$ is rationally connected if it has a completion $\ol{X}$ 
such that $\ol{X}_{\ol{\k}}$ is rationally connected in the usual sense \cite[Definition IV.3.2]{Kollarrat}. 
When $X$ is not proper 
this is different from the definition given in \cite{Kollarrat}, but it gives us a birational property 
more convenient for our purposes.
Note that by definition a rationally connected variety is always geometrically irreducible.

For any
variety $X$, there exists a rational dominant map $\pi\colon X \dashrightarrow B$
called the maximal rationally connected (MRC) fibration \cite[IV.5]{Kollarrat}, due to Campana and Koll\'ar--Miyoka--Mori which is constructed as follows. 
Fix a compactification $\ol{X}$ of $X$.
Over the algebraic closure $\ol{\k}$, $B_{\ol{\k}}$ is the quotient, in the sense of Campana, of $\ol{X}(\ol{\k})$ by the equivalence relation generated by
$x \sim y$ if there is a rational curve passing through $x$ and $y$.
Then $X_{\ol{\k}} \dto B_{\ol{\k}}$ canonically descends to the field $\k$. 
Up to birational modifications, the MRC fibration is unique.

By the main result of Graber--Harris--Starr \cite{GHS},
the MRC fibration is characterized by the following two
properties 
\begin{enumerate}
    \item the generic fiber of $\pi$ is rationally connected; 
    \item $B$ is not uniruled.
\end{enumerate}
Note that both rational connectedness and uniruledness can be checked over the algebraic closure of $\k$~\cite[Remarks 4.2 and 4.22]{Debarre}.
Hence the MRC fibration is also preserved under field extensions, if we define MRC fibrations for reducible reduced schemes of finite type by taking disjoint union of the MRC fibrations of their irreducible components.

\begin{example}\label{ex:P1-C-R}
Let $X$ be a real curve defined as the restriction of scalars:
\[
X = \P^1_\C \to \Spec(\C) \to \Spec(\R).
\]
We claim that the MRC fibration of $X/\R$ is the morphism $X \to \Spec(\C)$. Indeed the only fiber $\P^1_\C$ is rationally connected and the base is not uniruled. Extending scalars we get the MRC fibration of a disjoint union of 
$\C$-varieties
\[
X_\C = \P^1_\C \sqcup \P^1_\C \to \Spec(\C) \sqcup \Spec(\C)
\]
\end{example}

Any birational map $X \dto X'$
induces
a birational map between the bases of the MRC fibrations
\cite[Theorem IV.5.5]{Kollarrat}, in particular we have
\begin{equation}\label{eq:BirX-MRC}
\Bir(X) = \Bir(\pi).
\end{equation}

If $X \dto B$ is an MRC fibration,
we define the \emph{rationally connected dimension} of $X$ to be 
\[
\RCdim(X) = \dim(X) - \dim(B).
\]
In other words, $\RCdim(X)$
is the maximal integer $d$ such that
a general point $x \in X(\ol{\k})$ is contained in a $d$-dimensional rationally connected subvariety of $X_{\ol{\k}}$. 
If $X$ is rationally connected, then $\RCdim(X) = \dim(X)$ and the converse holds when $X$ is geometrically integral; Example \ref{ex:P1-C-R} shows why this is a necessary assumption.

\begin{lemma}\label{lem:RC-fibers-MRC}
If $\psi\colon Y \to X$ is a dominant morphism with rationally connected generic fiber, and $\pi\colon X \dto B$ is the MRC fibration of $X$, then
$\pi \circ \psi$ is the MRC fibration of $Y$.
\end{lemma}

\begin{proof}
Since MRC fibrations are descended from the algebraic closure, we can assume that $\ol{\k}  = \k$. Let us show that the composition $\pi \circ \psi\colon Y \dto B$ satisfies conditions (1), (2) characterizing MRC fibrations.
First assume that $B = \Spec(\k)$; in this case $X$ is rationally connected, therefore $Y$ is also rationally connected by \cite[Corollary 1.3]{GHS}. Thus $Y \to \Spec(\k)$ is an MRC fibration.

In general, since $B$ is not uniruled (because $X \dto B$ is an MRC fibration), it suffices to check that the generic fiber of $Y \dto B$ is rationally connected. This follows by passing to $\Spec(\k(B))$ and using the special case considered above.
\end{proof}

\begin{theorem}\label{thm-van3fold}
If $X$ is any threefold over an algebraically closed field $\k$ of characteristic zero, then $c$ is identically zero on $\Bir(X)$. 
\end{theorem}

\begin{proof}
We use the MRC fibration $\pi \colon X \to B$, which we can
assume to be a smooth projective morphism,
with smooth $X$ and $B$.
General fibers of $\pi$ are rationally connected.

The proof relies on \eqref{eq:BirX-MRC}.
We have the following four possibilities for $\pi$, depending on the rationally connected dimension of $X$:

\begin{itemize}
    \item $\RCdim(X) = 3$ and $\dim B = 0$:
    Then $X$ is rationally connected.
    The result holds by 
    the last claim in \cite[Proposition 2.6]{BirMot}.

    \item $\RCdim(X) = 2$ and $B$ is a curve, and $\pi$ has relative dimension two. We have $c_\hor(\phi) = 0$ by Vanishing I
    (Proposition \ref{prop:vanishingI}).
    Since general fibers of $\pi$ are smooth rationally connected surfaces, they are rational varieties. Thus
    $\pi$ is birationally trivial in codimension $1$ by Example \ref{ex:bir-triv-curve} and $c_\ver(\phi) = 0$ by Vanishing III (Corollary \ref{cor:vanishingIII}).
    \item $\RCdim(X) = 1$ and $B$ is a non-ruled 
    surface. A general fiber of $\pi$ is a smooth rational curve. 
    In this case we have $c_\hor(\phi) = 0$ by 
    Lemma \ref{lem:rel-1}.
    Note that $\pi$ is birationally trivial in codimension $1$ by 
    Example \ref{ex:bir-triv-surf}.
    Therefore $c_\ver(\phi) = 0$
    again by Vanishing III.
    \item $\RCdim(X) = 0$ and $\pi$ is a birational isomorphism, that is $X$ is not uniruled, hence by running MMP~\cite{Mori88} we can assume it is a K-nef threefold with $\Q$-Gorenstein terminal singularities. In this case $\phi$ and $\phi^{-1}$ have no exceptional divisors \cite[Corollary 3.54]{KollarMori}, hence $c(\phi) = 0$.
\end{itemize}
\end{proof}

Let us consider a group homomorphism 
\[
\MRC \colon \Burn_*(\k) \to \Burn_*(\k)
\]
which sends
a birational class $[X]$ to the class of its MRC base $[B]$. Note that it is not a graded homomorphism as it can lower the degree of a class, however it does preserve the subgroups defining an increasing filtration 
\[
\Burn_{\le n}(\k) := \bigoplus_{m=0}^n \Burn_m(\k).
\]

\begin{example}\label{ex-MRCd-2} 
    Since the exceptional divisors of a birational automorphism
    of a smooth proper variety 
    are ruled~\cite[Theorem VI.1.2]{Kollarrat}, 
    by the resolution of singularities
    (recall that $\chr(\k) = 0$), we have
    $$c(\Bir(X)) \in [\bP^1] \cdot \Burn_{\dim(X)-2}(\k)$$
    for any $\k$-variety $X$.
    In particular,
    \begin{equation}
    \label{eq:MRC-exceptional}
    \MRC(c(\Bir(X))) \subset \Burn_{\le \dim(X)-2}(\k).
    \end{equation}
\end{example}

\begin{proposition}\label{prop:RCfib}
Assume that $\k$ is a field of characteristic zero.
Let $\pi_i\colon X_i \to B_i$ for $i = 1,2$ be flat surjective morphisms 
between normal varieties with rationally connected fibers.
Let $\phi \in \Bir(\pi_1,\pi_2)$ 
and $\sigma \colon B_1 \dto B_2$ be the induced map.
Then we have
$$\MRC(c_\ver(\phi)) = \MRC(c(\sigma)).$$
\end{proposition}

\begin{proof}
Since $\pi_i$ is flat and surjective, and the fibers are irreducible (because they are rationally connected) vertical prime divisors in $X_i$ are in bijection with prime divisors in $B_i$.
By Proposition~\ref{pro:cvercons}, we have
\begin{equation}\label{eqn-cverproof}
    c_\ver(\phi) = \sum_{\zeta \in B_1^{(1)} \cup B_2^{(1)}}
\left([\pi_2^{-1}(\ol{\zeta}^{B_2})] - [\pi_1^{-1}(\ol{\zeta}^{B_1})] \right).
\end{equation}
As the fibers of $\pi_i$ are rationally connected, using Lemma \ref{lem:RC-fibers-MRC}
we obtain
$$
\MRC\left([\pi_i^{-1}(\ol{\zeta}^{B_i})] \right)
= \MRC\left([\ol{\zeta}^{B_i}] \right).
$$
Applying $\MRC(-)$ to both sides of~\eqref{eqn-cverproof}, together with~\eqref{eq:c-valuations-def},
we get the result.
\end{proof}

The following corollary extends Example~\ref{ex-MRCd-2}.

\begin{corollary}\label{cor:ver-MRC}
Assume that $\k$ is a field of characteristic zero.
    Let $\pi\colon X \to B$ be a dominant morphism
    of relative dimension $d$. Suppose that
    the generic fiber has rationally connected geometric irreducible components.
    We have
     \[
     \MRC(c_\ver(\Bir(\pi))) \subset \Burn_{\le \dim(X)-d-2}(\k).
     \]
\end{corollary}

\begin{proof} 
By Corollary~\ref{cor-SFchor}, we can assume that the geometric generic fiber of $\pi$ is irreducible.
Up to birational modification, we can
assume that $X$ and $B$ are smooth and that
$\pi$ is smooth and proper, in particular surjective.
Since rationally connectedness is an open property for smooth proper morphisms \cite[Theorem IV.3.11]{Kollarrat}
we can also assume that every fiber of $\pi$ is  
rationally connected.
The result now follows from Proposition~\ref{prop:RCfib},
because by \eqref{eq:MRC-exceptional}
$$\MRC(c(\Bir(B))) \subset \Burn_{\le \dim(B) - 2}(\k).
$$ 
\end{proof}

\section{Unboundedness of the image of $c$ and applications}\label{sec-ubd}

In this section, $\k$ is a field of characteristic zero. The main results are Theorem \ref{thm-unbounded-c-B}
and its Corollaries 
\ref{cor:abelianizations}
and \ref{cor:gener-CB-SB}.
The other results in this section are technical steps required in the proof of 
Theorem \ref{thm-unbounded-c-B}.
These include constructing  an unbounded sequence of  elliptic fibrations with prescribed properties (Lemma \ref{lem:B-diagram} and Proposition \ref{prop:constr-J}).

\subsection{Unboundedness}

\begin{definition}\label{def-unbounded} 
We say that a subgroup $H \subset \Burn_*(\k)$
is geometrically bounded if there is
a flat proper morphism $\cD \to T$ of $\ol{\k}$-schemes of finite type such that
the image of $H$ in $\Burn_*(\ol{\k})$
is contained in a subgroup 
generated by $[\cD_t]$, $t \in T(\ol{\k})$.
\end{definition}

\begin{theorem}\label{thm-unbounded-c-B}
Let $\k$ be a field of characteristic zero.
Assume that $X$ is an $n$-dimensional variety birational to $B \times \P^3$
for some geometrically integral variety $B$ of positive dimension (for example $X = \P^n$ with $n \ge 4$).
Then the image
$$\Ima\left(\Bir(X) \overset{c}{\longrightarrow} \Burn_{n-1}(\k) \overset{\MRC}{\longrightarrow} \Burn_{\le n-1}(\k)\right)$$
contains
a geometrically unbounded subgroup of $\Burn_{n-2}(\k)$.
\end{theorem}

The assumption $\dim B > 0$ 
in Theorem~\ref{thm-unbounded-c-B} 
is necessary by Theorem~\ref{thm-van3fold}.
Note that the MRC base dimension $n-2$ in Theorem~\ref{thm-unbounded-c-B} 
is the maximal possible by Example~\ref{ex-MRCd-2}.
To prove Theorem \ref{thm-unbounded-c-B}, we will use an unbounded sequence of elliptic fibrations over $B$.
We first prove some preliminary results
under the assumptions of Theorem \ref{thm-unbounded-c-B}.

We start by recalling some basic facts about Iitaka fibrations \cite[Theorem 6.11]{UenoClassAlgVar}, \cite[\S2]{LazarsfeldPos}. Let $X$ be a smooth projective variety of Kodaira dimension $\kappa(X) = \dim(X) - 1$. In this case we say that $X$ has \emph{Kodaira codimension $1$}. The so-called Iitaka fibration, defined 
by a linear system $|K_X^{\otimes m}|$ for a sufficiently divisible positive $m$ \cite[\S2]{LazarsfeldPos}, 
is a rational dominant map
\[
\pi\colon X \dto Z
\]
whose generic fiber is a curve of genus $1$, namely, $\pi$ is a rational elliptic fibration.
Note that even though the Iitaka fibration in \cite{UenoClassAlgVar}, \cite{LazarsfeldPos} is  defined over $\k = \C$, due its canonical nature, it automatically descends to 
any ground field of characteristic zero.

The Iitaka fibration is a birational invariant of $X$ in the sense that every birational map induces a birational map between the Iitaka fibrations. 
Furthermore, every rational dominant map $X \dto Z'$ whose generic fiber is a curve of genus $1$ is birational to the Iitaka  fibration of $X$ \cite[Theorem 6.11(5)]{UenoClassAlgVar}, in other words $X$ has an essentially unique structure of a rational elliptic fibration.

Thus for any smooth projective variety of Kodaira codimension $1$ we obtain a canonically defined $j$-invariant map $j_X \colon Z \dto \P^1$.
We refer to the Stein degree (see Definition \ref{def-RStein}) of $j_X$ as the \emph{Iitaka--Stein degree} of $X$. The Iitaka--Stein degree provides a simple way to measure the complexity of Kodaira codimension $1$ varieties.

If $X$ is an arbitrary integral variety, then by the Kodaira dimension, the Iitaka fibration, and the Iitaka--Stein degree, we mean the corresponding invariants for any smooth projective model $\wt{X}$ of $X$.

\begin{prop}\label{pro-bdSdegell}
    Let $\pi \colon \cX \to T$ be a projective morphism between varieties over an algebraically closed field $\k$ of characteristic zero.
    Let $U \subset T(\k)$ be the locus parameterizing 
    fibers
    $X_t \cnec \pi^{-1}(t)$ which are 
    integral varieties of
    Kodaira codimension $1$.
    Then the Iitaka--Stein degrees of $X_t$, $t \in U$
    are bounded above.
\end{prop}

\begin{proof}
    We can remove the closed subscheme of $T$ 
    parameterizing fibers which are not integral,
    and thus assume that every fiber of 
    $\pi$ is reduced and irreducible.
    Then there exists a finite stratification $T = \bigsqcup_i T_i$
    such that over each $T_i$, the family $\pi$ 
    has a simultaneous resolution
    of singularities $\wt{\cX_i} \to T_i$.
    As the Kodaira dimension and the Iitaka--Stein degree are birational invariants,
    working stratum by stratum,
    we can therefore assume that 
    $\pi$ is smooth and projective.
    Since the Kodaira dimension is locally constant 
    in characteristic zero~\cite{SiuPg},
    we can assume that 
    every fiber $X_t \cnec \pi^{-1}(t)$ has Kodaira codimension $1$.

    By~\cite[Theorem 2]{KawamataBook}, which improves~\cite{BCHM},
    the $\cO_T$-algebra 
    $\bigoplus_{m = 0}^\infty \pi_*\go_{\cX/T}^{\otimes m}$
    is finitely generated.
    Define 
    $$\cZ \cnec 
    \cProj \left(\bigoplus_m \pi_*\go_{\cX/T}^{\otimes m}\right),$$
    we thus have a factorization
    $$
     \xymatrix{
    \cX \ar[dr]_\pi \ar@{-->}[r]  &  \cZ  \ar@{-->}[r] \ar[d] & \P^1_T \ar[dl] \\
     & T & }
    $$
    such that over every $t \in T$,
    we get an Iitaka fibration 
    $\cX_t \dto \cZ_t$ of $\cX_t$ and its $j$-map $j_{\cX_t}\colon \cZ_t \dto \P^1$.
    The boundedness of the Stein degrees for these $j$-maps
    follows from Lemma~\ref{lem-bdStein} applied to the right triangle in the diagram.
\end{proof}

Our next goal is to construct an unbounded sequence of Kodaira codimension one varieties elliptically fibered over a fixed base $B$, see Proposition \ref{prop:constr-J}. The first step in that direction is the following.

\begin{lem}\label{lem:B-diagram} 
Let $B$ be a smooth projective geometrically integral
variety of dimension $n > 0$ over a field $\k$ of characteristic zero.
Given any finite morphism $j \colon \P^1_t \to \P^1_j$
and an integer $d \ge 1$, there exist a finite morphism
$g\colon B \to \P^n$, a finite morphism
$f_d \colon \P^1_u \to \P^1_t$ of degree at least $d$,
and a linear projection $p\colon \P^n \dto \P^1_u$
such that
for the composition
\begin{equation}\label{cd-Bdiagram}
    \xymatrix{
B \ar[r]^g & \P^n \ar@{-->}[r]^p 
& \P^1_u \ar[r]^{f_d} & \P^1_t \ar[r]^j &\P^1_j,
}
\end{equation}
the automorphism group $\Aut(\ol{\k}(B)/\ol{\k}(j))$ 
is trivial.
\end{lem}
\begin{proof}

We first consider the case when $B$ 
is a curve. In this case $p$ is the identity 
and $f_d$ is an arbitrary finite morphism of degree $d$.
We take $g$ to be a Lefschetz pencil, which in dimension one is the same as a simple covering,
that is we assume that
$g$ has simple ramification and at most one ramification point over every point in $\P^1_u(\ol{\k})$.
We can assume in addition  that
the branch locus $\Sigma \subset \P^1_u{(\ol{\k})}$ of $g$ satisfies $|\Sigma| > 2g(B) + 2$.
Furthermore we can make a choice of $g$, 
such that $jf_d|_{\gS}$ is injective and $jf_d$ is a simple covering
in the neighborhood of $j(f_d(\Sigma))$.
It follows that over each point $x \in j(f_d(\Sigma))$,
the finite cover $B \to \P^1_j$ has exactly one ramification point.
Therefore
any element of $\Aut(\ol{\k}(B)/\ol{\k}(j))$ 
fixes $|\Sigma|$ ramification points in $B$. Since $|\Sigma| > 2g(B) + 2$, we see that there are no nontrivial automorphisms by the Lefschetz fixed point theorem.    

Now assume that $\dim(B) \ge 2$.
Take a very ample line bundle $L$ on $B$ such that 
    $L \otimes \go_B$ is also very ample; see~\cite[Example 1.2.10]{LazarsfeldPos} for the existence of $L$.
    By adjunction, 
    a smooth member $D$ of $|L|$ has ample canonical class.

    Let $g\colon B \to \bP^n$ be the finite morphism 
    defined by a general linear system in $|L|$ of dimension $n$.
    The composition $B \xto{g} \bP^n \overset{p}\dto \P^1_u$
    is then defined by a general
    pencil $|L|'$ in $|L|$.
    Blowing up the base locus of the pencil $|L|'$, we obtain a resolution $\ti{f}\colon \wt{B} \to \P^1_u$ of $B \dto \P^1_u$.

    We show that
    $\Bir({B}_{\ol{\k(\P^1_u)}})$ is trivial, 
    which implies that 
    $\Aut(\ol{\k}(B)/\ol{\k}(u))$
    is trivial.
    By extending $\k$, we can assume that it is uncountable and algebraically closed, in which case there exists an isomorphism $\ol{\k(\P^1)} \simeq \k$ which identifies the geometric generic fiber ${B}_{\ol{\k(\P^1_u)}}$ with the very general fiber $D \subset B$ of $\wt{f}$, see e.g.~\cite[Lemma 2.1]{VialVG}.
    As $\dim B \ge 2$, 
    $D$ 
    is irreducible~\cite[Theorem 3.3.1]{LazarsfeldPos}.
    Since the canonical bundle of $D$ is ample,
    we have $\Aut(D) = \Bir(D)$, see e.g.~\cite[Corollary 1.2]{CheltsovCremona}.
    We conclude using~\cite[Theorem 1.4]{lyuZhang}
    that 
    $$\Bir({B}_{\ol{\k(\P^1)}}) \simeq \Bir(D) = \Aut(D)$$ 
    is trivial.

Arguing like in the first part of the proof, we can construct
a finite morphism
$f_d \colon \P^1_u \to \P^1_t$ of degree 
at least $d$ such that
$\Aut(\ol{\k}(u)/\ol{\k}(j))$ is trivial; precisely, in the notation of the first part of the proof we take $B = \P^1_u$ and $g f_d$ to be $f_d$ in the current case.
In the tower of extensions from~\eqref{cd-Bdiagram}:
$$\ol{\k}(j) \subset \ol{\k}(t)  \subset \ol{\k}(u) \subset \ol{\k}(B),$$
since $\ol{\k}(u)$ is the algebraic closure of $\ol{\k}(j)$ in $\ol{\k}(B)$ (because $B$ is geometrically integral over
$\k(u)$ due to $\dim B \ge 2$), 
the subfield $\ol{\k}(u) \subset \ol{\k}(B)$ is preserved under
$\Aut(\ol{\k}(B)/\ol{\k}(j))$, and
we have an exact sequence
$$1 \to \Aut(\ol{\k}(B)/\ol{\k}(u)) 
\to \Aut(\ol{\k}(B)/\ol{\k}(j)) \to \Aut(\ol{\k}(u)/\ol{\k}(j)).$$
As both $\Aut(\ol{\k}(B)/\ol{\k}(u))$ and $\Aut(\ol{\k}(u)/\ol{\k}(j))$
are trivial, so is $\Aut(\ol{\k}(B)/\ol{\k}(j))$.
\end{proof}

\begin{proposition}\label{prop:constr-J}
Let $B$ be as in Lemma~\ref{lem:B-diagram}. 
There exists a sequence $\xi_d\colon J_d \to B$, $d \ge 1$ of nonisotrivial elliptic fibrations with the following properties:
\begin{enumerate}
    \item $J_d$ has Kodaira codimension $1$ and 
    $\xi_d$ is an Iitaka fibration for $J_d$.
    \item The Iitaka--Stein degree of $J_d$ is at least $d$.
    \item The Mordell--Weil group of rational sections of $J_d \to B$ contains $5$-torsion subgroup $\Z/5$.
    \item We have 
    $\Bir(J_{d,\ol{\k}}) = \Bir(J_{d,\ol{\k}}/B_{\ol{\k}})$; namely,
    every birational automorphism of 
    $J_{d,\ol{\k}}$ preserves the elliptic fibration structure and descends to $\Id_{B_{\ol{\k}}}$ through $\xi_d$.
\end{enumerate}
\end{proposition}

\begin{proof}
The idea is to start with an appropriate elliptic fibration over 
$\P^1_t$ 
and pull it back to $B$ under the maps  in \eqref{cd-Bdiagram}.
By~\cite[Theorem 5.1]{CoxParry}, 
there exists a elliptic curve $E/\Q(t)$ 
with  nonconstant $j$-invariant
and
such that $\Z/5 \subset E(\Q(t))$.
We make a scalar extension of $E/\Q(t)$ to $\k(t)$ 
and take a model  
$\xi^{\P^1} \colon J \to \P^1_t$. 
By construction the group of sections
of $\xi^{\P^1}$ contains $\Z/5$.

We now use the maps constructed in Lemma \ref{lem:B-diagram},
applied to the $j$-invariant map $j \colon \P^1_t \to \P^1_j$.
There exists a finite morphism $F_d$  of the same degree as $\deg f_d \ge d$ which fits into a commutative diagram
\begin{equation*}
    \xymatrix{
 \P^n \ar[r]^{F_d} \ar@{-->}[d]^p & \P^n \ar@{-->}[d]^p   \\ 
 \P^1 \ar[r]^{f_d} & \P^1  .
}
\end{equation*}

Fix an elliptic fibration
$$\xi^{\P^n} \colon \cE \to \P^n, 
$$ 
birational to the pullback of 
$\xi^{\P^1}$ with respect to $p \colon \P^n \dto \P^1$.
Let $\xi_d^{\P^n}$ and $\xi_d \colon J_d \to B$ be the pullbacks of $\xi^{\P^n}$ under $F_d$ and $F_d\circ g$ respectively. 
By construction, 
these are elliptic fibrations that satisfy 
(3) and whose $j$-invariant map
has Stein degree $\ge d$.
The latter implies (2) once (1) is proved.

Now we prove (1).
It suffices to prove the statement for
$\xi_d \times_\k K$,
where $K$ is an algebraically closed extension of $\k$ 
containing $\C$.
Fix a Weierstrass model
(see~\cite[Definition 1.1 and Theorem 2.1]{NakayamaWM})
$$\zeta^{\P^n} \colon W(\cL,a,b) \to \P^n, 
\text{ with }\cL \in \Pic(\P^n), \; a \in \Gamma((\cL^\vee)^{\otimes 4}), \;
\; b \in \Gamma((\cL^\vee)^{\otimes 6})
$$ 
birational to $\xi^{\P^n} \times_\k K$.
Let $\zeta_d^{\P^n}$ and $\zeta_d \colon J_d \to B$ be the pullbacks of $\zeta^{\P^n}$ under $F_d$ and $F_d\circ g$ respectively. 
Since $\xi^{\P^1}$ is not isotrivial, so is $\zeta^{\P^n}$,
which implies $\cL \not\simeq \cO_{\P^n}$.
As $a$
is a nonzero section of $(\cL^\vee)^{\otimes 4}$ (again because $\zeta^{\P^n}$ is not isotrivial)
the line bundle $\cL^\vee$ is ample.
We can assume
that $d$ is large enough so that $\go_{\P^n} \otimes F_d^*\cL^\vee$ is ample. As $\zeta_d^{\P^n}$ is the Weierstrass model
$$W(F_d^*\cL, F_d^*a, F_d^*b) \to \P^n,$$
if follows from the canonical bundle formula for Weierstrass fibrations~\cite[(1.2)(2)]{NakayamaWM}
that for such $d$,
property (1) holds for $\zeta_d^{\P^n}$, and hence also for $\zeta_d$ because $g$ is finite. 
As $\xi_d \times_\k K$ is birational to $\zeta_d$, (1) is proved.

Finally we prove (4).
Since $\xi_d$ is an Iitaka fibration of $J_d$, we have $\Bir(J_{d,\ol{\k}}) = \Bir(\xi_{d,\ol{\k}})$.
For any $\phi \in \Bir(\xi_{d,\ol{\k}})$, 
the map $\gs \in \Bir(B_{\ol{\k}})$ induced by $\phi$
satisfies $\sigma \in \Aut(\ol{\k}(B)/\ol{\k}(j))$. However, the latter group is trivial by Lemma \ref{lem:B-diagram}, so $\sigma = \id_B$.
\end{proof}

\begin{corollary}\label{cor:P3-kB}
Let $\k$ be a field of characteristic zero and $B$ a geometrically integral variety over $\k$ of positive dimension.
Take any $d \ge 1$.
There exist $C$ and $C'$ 
which are 
torsors over  the generic fiber
$J_{d,\k(B)}$ of the elliptic fibration $\xi_d$
from Proposition \ref{prop:constr-J}
and
a birational map $\phi \in \Bir(\P^3_{\k(B)})$
such that 
$$c(\phi_{\ol{\k}(B)}) = ([C_{\ol{\k}(B)}] - [C'_{\ol{\k}(B)}]) \cdot [\P^1_{\ol{\k}(B)}] \ne 0.$$
\end{corollary}

\begin{proof}
We can assume that  $B$ smooth and projective.
We  argue as in the proof of~\cite[Lemma 3.8]{BirMot}.
To simplify the notation,
let us write 
$E = J_{d,\k(B)}$ for a fixed $d \ge 1$.

By Proposition~\ref{prop:constr-J}, $E$ satisfies the assumptions in Proposition \ref{prop:E-torsors} with $p = 5$.
Take the $E$-torsor $C$ constructed in
Proposition \ref{prop:E-torsors}
and let $\alpha \in H^1(\k(B), E)[5]$ be the corresponding class.
Let $C' := \Pic^2(C)$, i.e. we take
the $E$-torsor corresponding to $2\alpha$.
Since $\xi_d$ is not isotrivial, 
in particular its $j$-invariant is not constant $1728$, $C_{\ol{\k}(B)}$ and $C'_{\ol{\k}(B)}$ are not isomorphic as curves (not just as $E$-torsors) by \cite[Lemma 2.7]{ShinderZhang}.    

By \cite[\S 3.2]{BirMot} there exists a birational map $\phi \in \Bir(\P^3_{\k(B)})$
such that
$$c(\phi_{\ol{\k}}) = ([C_{\ol{\k}(B)}] - [C'_{\ol{\k}(B)}]) \cdot [\P^1_{\ol{\k}(B)}] \in \Burn_2(\ol{\k}(B)),$$ 
which is nonzero because $C_{\ol{\k}(B)}$ and $C'_{\ol{\k}(B)}$ are not stably birational.
\end{proof}

\begin{proof}[Proof of Theorem \ref{thm-unbounded-c-B}]

We can assume that $X = B \times \P^3$ with $B$ smooth and projective.
Fix $d \ge 1$ and consider $\xi_d \colon J \cnec J_d \to B$ defined in Proposition \ref{prop:constr-J}. 

Let $\pi\colon \P^3 \times  B \to B$
be the second projection and
 $\phi'$ be the same birational map as $\phi$ from Corollary \ref{cor:P3-kB}, 
but considered in $\Bir(\P^3 \times B / B)$.
We have a commutative diagram
\[
\xymatrix{\P^3 \ar[dr] \times B \ar@{-->}[rr]^{\phi'} & & \P^3 \times B \ar[dl] \\ & B & }
\]
which restricts to $\phi$ on the generic fiber $\k(B)$.

Let $Y$ and $Y'$  be smooth projective  models of $C$ and $C'$ over $B$.
By Corollary 
\ref{cor:vanishingII} we have $c_\ver(\phi') = 0$
so that
\[
c(\phi')  = c_\hor(\phi')  = [\P^1 \times Y] - [\P^1 \times Y'].
\]

Let us show that $\MRC(c(\phi'_{\ol{\k}})) \ne 0$. 
There is a dominant morphism $C \to J^0(C)$ over $\k(B)$ (e.g. multiplication by $5$),
hence $Y$ dominates $J$. In particular using
Proposition~\ref{prop:constr-J}(1) we have
\[
\dim Y - 1 \ge \gk(Y) \ge \gk(J) = \dim J - 1,
\]
so necessarily $Y \to B$ and 
$Y_{\ol{\k}} \to B_{\ol{\k}}$ are the Iitaka fibrations~\cite[Theorem 6.11]{UenoClassAlgVar}.
The same holds for $Y'$.
In particular, both $Y$ and $Y'$ are not uniruled, and if 
$\P^1_{\ol{\k}} \times Y_{\ol{\k}}$ and $\P^1_{\ol{\k}} \times Y'_{\ol{\k}}$ are birational,
then we have a birational map on the MRC bases
$\psi \colon Y_{\ol{\k}} \dto Y'_{\ol{\k}}$, 
which descends to $\gs \in \Bir(B_{\ol{\k}})$ through the Iitaka fibrations: 
    \[\xymatrix{
    Y_{\ol{\k}} \ar[d] \ar@{-->}[r]^\psi & Y'_{\ol{\k}} \ar[d]\\
    B_{\ol{\k}} \ar@{-->}[r]^\sigma
    & B_{\ol{\k}}. \\
    }\]
    It also induces a birational self-map $J^0(\psi) \in \Bir(\xi_{\ol{\k}})$ which descends to $\sigma \in \Bir(B_{\ol{\k}})$. 
    Thus $\sigma$ is the identity by Proposition \ref{prop:constr-J}(4), so the generic fibers $C$ and $C'$ are isomorphic over $\ol{\k}(B)$,
    which contradicts Corollary~\ref{cor:P3-kB}. This shows that
    \[
    \MRC(c(\phi'_{\ol{\k}})) = [Y_{\ol{\k}}] - [Y'_{\ol{\k}}] \ne 0.
    \]

Finally, when we increase $d \ge 1$, this construction produces infinitely many classes in the image $\MRC(c(\Bir(\P^3 \times B)))$ and this subgroup is geometrically unbounded
by 
Proposition~\ref{pro-bdSdegell}
because the Iitaka--Stein degree of $Y_{\ol{\k}}$,
which is equal to that of $J_d$ (since their $j$-maps are the same), is unbounded in $d$ by 
Proposition~\ref{prop:constr-J}(2).
\end{proof}

\subsection{Applications}

We start with an immediate consequence of 
Theorem \ref{thm-unbounded-c-B} for abelianizations of birational automorphism groups.
We assume that $\k$ is a field of characteristic zero. 

\begin{corollary}\label{cor:abelianizations}
If $B$ is any geometrically integral variety, 
then for any $k \ge 3$ 
the canonical morphism between abelianizations
$
\Bir(\P^k \times B)^\ab \to
\Bir(\P^{k+1} \times B)^\ab 
$
is not surjective.
\end{corollary}
\begin{proof}
Let $n = \dim(B) \ge 0$.
Let 
$$c_k \colon \Bir(\P^k \times B)^\ab \to  \Burn_{k+n-1}(\k)$$
be the homomorphism induced by the motivic invariant $c$.
By Example~\ref{ex-MRCd-2},
we have a commutative diagram
\[
\xymatrix{
\Bir(\P^k \times B)^\ab \ar[d] \ar[r]^-{c_k} & \Image(c_k) \ar[d]^{\times \P^1} 
\ar[r]^{\MRC\;\;\;\;\;\;\;\;\;\;} & {\Burn_{\le k+n-2}(\k)}  \ar@{^{(}->}[d]\\
\Bir(\P^{k+1} \times B)^\ab \ar[r]^-{c_{k+1}}  & \Image(c_{k+1}) \ar[r]^{\MRC\;\;\;\;\;\;} & {\Burn_{\le k+n-1}(\k)} \\
}
\]

By   Theorem \ref{thm-unbounded-c-B}, applied to $\P^3 \times (\P^{k-2} \times B)$,
the image of $c_{k+1}$ contains elements whose base of the MRC fibration has dimension $k+n-1$.
Thus the left vertical map is not surjective.
\end{proof}

If $\pi\colon X \dto S$ is a rational dominant map, then we refer to the subgroup $\Bir(\pi) \subset \Bir(X)$ as birational maps preserving $\pi$.
For example, we can consider a linear projection $\pi\colon \P^n \dto \P^{n-1}$
and maps 
$\phi$
fitting 
into commutative diagram
    \[\xymatrix{
    \P^n \ar@{-->}[d]^\pi \ar@{-->}[r]^\phi & \P^n \ar@{-->}[d]^\pi \\
    \P^{n-1} \ar@{-->}[r]^\sigma
    & \P^{n-1} \\
    }\]
which are called
\emph{Jonqui\`eres map} \cite{PanSimis}.
Pan and Simis have asked whether Cremona groups can be generated by linear automorphisms and de Jonqui\`eres maps \cite[p. 925]{PanSimis}.
This has been answered in~\cite[Theorem C]{BLZ} in the negative.
We have the following more general statement.

\begin{corollary}\label{cor:gener-CB-SB}
Let $X$ be birational to  $\P^3 \times B$ for a positive-dimensional geometrically integral variety $B$. 
Then $\Bir(X)$
is not generated by pseudo-regularizable maps and birational maps preserving a 
conic bundle or a rational surface fibration.
\end{corollary}

Here by a conic bundle (resp. rational surface fibration) structure we mean a rational dominant map $\pi\colon X \dto B$ whose generic fiber is a conic (resp. a geometrically rational surface).

\begin{proof}
By Theorem 
\ref{thm-unbounded-c-B}, $\MRC(c(\Bir(X)))$ contains nonzero classes of dimension $\dim(X) - 2$. 
The invariant $c$ vanishes on pseudo-regularizable maps by \cite[Lemma 4.3]{BirMot}.
Let $\pi\colon X \dto B$ be a conic bundle or a rational surface fibration
and $\phi \in \Bir(\pi)$.
Then $c_\hor(\phi) = 0$ by Proposition~\ref{prop:vanishingI},
and $\MRC(c_\ver(\phi))$ is generated by classes of dimension $\le \dim(X) - 3$ by Corollary~\ref{cor:ver-MRC}. Thus all these types of elements can not generated $\Bir(X)$.
\end{proof}

\begin{example}
In \cite{BSY} the authors construct nontrivial homomorphisms from Cremona groups, based on type II links between Severi--Brauer surface fibrations \cite[Theorem 6.2.4]{BSY}. By Corollary \ref{cor:gener-CB-SB} these elements do not generate the respective groups of birational self-maps.
\end{example}

\appendix
\section{Constructing elliptic torsors of prescribed prime index}

The following result produces torsors which we use to construct birational self-maps of $\P^3_{\k(B)}$ in Corollary \ref{cor:P3-kB};
the construction of torsors of prescribed index is a variation on a theme by Lang--Tate \cite[Theorem 7]{Lang-Tate} and Clark--Lacy \cite[Theorem 1.6]{ClarkLacy}.
Recall that a curve $C$ of genus $1$  has  index $p$, if $C$ has no rational points and admits closed points of degree $p$.

\begin{prop}\label{prop:E-torsors} 
Let $\k$ be a field of characteristic zero.
Let $B$ be a geometrically integral $\k$-variety of dimension $n > 0$. 
    Let $p$ be a prime number and 
    let $E$ be an elliptic curve over $\k(B)$
    whose $j$-invariant is not in $\k$.
    Suppose that $\Z/p \subset E(\k(B))$.
    Then there exist infinitely many $E$-torsors $\{C_i\}_{i \in \N}$
    such that $C_{i, \ol{\k}(B)}$ are pairwise non isomorphic (as $\ol{\k}(B)$-varieties) and that 
    each $C_{i, \ol{\k}(B)}$ has index $p$.
\end{prop}

\begin{proof}
Recall that isomorphism classes of $E$-torsors
are parametrized by elements of the Galois cohomology group $H^1({\k}(B),E)$.
Consider the maps between the short exact sequences induced by the Kummer sequence
  $$
\xymatrix{
0 \ar[r] & \frac{E(\k(B))}{pE(\k(B))} \ar[d]  \ar[r] &
H^1(\k(B),E[p])  \ar[d]  \ar[r] & H^1(\k(B),E)[p] \ar[d]  \ar[r] & 0 \\
0 \ar[r] & \frac{E(\ol{\k}(B))}{pE(\ol{\k}(B))} \ar[r] &
H^1(\ol{\k}(B),E[p])   \ar[r]^{\gb} & H^1(\ol{\k}(B),E)[p]  \ar[r] & 0
}
$$
Since the $j$-invariant $j_E \in \k(B)$ of $E$ is not in $\k$ and $B$ is geometrically integral,
so that $\ol{\k} \cap \k(B) = \k$,
$j_E$ is not in $\ol{\k}$ neither.
So both $\frac{E(\k(B))}{pE(\k(B))}$ and $\frac{E(\ol{\k}(B))}{pE(\ol{\k}(B))}$ are finite (see e.g.~\cite[Example 2.2]{Bonrad-K/k}). 

Let $Q$ denote the quotient of $E[p]$ by $\Z/p$ as group schemes. 
Then we have a commutative diagram with exact rows:
 $$
\xymatrix{
Q(\k(B)) \ar[d]  \ar[r] &
H^1(\k(B),\Z/p)  \ar[d]  \ar[rd]^{\delta} \ar[r] & H^1(\k(B),E[p]) \ar[d]  \\
Q(\ol{\k}(B)) \ar[r] &
H^1(\ol{\k}(B),\Z/p)   \ar[r] & H^1(\ol{\k}(B),E[p]) 
}
$$
The vertical arrow in the middle is isomorphic to
$$\Hom(G_{\k(B)}, \Z/p) \to \Hom(G_{\ol{\k}(B)}, \Z/p),$$
where $G_\F$ denotes the absolute Galois group of a field $\F$.
By Lemma~\ref{lem:Galois}, the image of this map is infinite.
It follows that $\gd$ has infinite image 
as $Q(\ol{\k}(B))$ is finite.
Since $\ker(\gb) \simeq \frac{E(\ol{\k}(B))}{pE(\ol{\k}(B))}$,
which is finite,
it follows that the image of the composition
$$\Hom(G_{\k(B)}, \Z/p) \simeq H^1(\k(B),\Z/p)
\overset{\delta}\to
H^1(\ol{\k}(B),E[p]) \xto{\gb} H^1(\ol{\k}(B),E)[p]$$
is infinite, which gives rise to infinitely many $E$-torsors $C_i$ which are still non isomorphic as $E_{\ol{\k}(B)}$-torsors.
There are only finitely many $E_{\ol{\k}(B)}$-torsor structures on a fixed curve of genus $1$ \cite[Exercise 10.4]{SilvermanAEC},
hence after removing repetitions we can assume that $C_{i, \ol{\k(B)}}$ are pairwise non isomorphic curves.

Finally let us show that every nonzero class
in the image $\Image(H^1(\ol{\k}(B),E[p])   \to H^1(\ol{\k}(B),E)[p])$ has index $p$, namely it splits by some degree $p$ extension. 
Take any element
$\alpha \in \Hom(G_{\ol{\k}(B)}, \Z/p)$ with nonzero image in $H^1(\ol{\k}(B),E)[p]$. 
By Galois theory $\alpha$ defines a degree $p$ extension $L/\ol{\k}(B)$ and by construction $\alpha_L = 0$. Thus the same holds for the image of $\alpha$ in $H^1(\ol{\k}(B),E)[p]$.
\end{proof}

The following lemma was used in the proof of Proposition \ref{prop:E-torsors}. It is a variant of 
the inverse Galois problem for $\Z/p$.

\begin{lem}\label{lem:Galois}
The image of the map
\begin{equation}\label{map-GalRes}
\Hom(G_{\k(B)}, \Z/p) \to \Hom(G_{\ol{\k}(B)}, \Z/p)  \end{equation}
is infinite.    
\end{lem}

The fact that $\Hom(G_{\k(B)}, \Z/p)$ is infinite, in other words that $\k(B)$ admits infinitely many cyclic Galois extensions of degree $p$ is well-known \cite[\S16]{MR2445111}, due to the fact that $\k(B)$ is a so-called Hilbertian field. It is however not immediately clear from the constructions in \cite{MR2445111} whether the appearing extensions do not become isomorphic after passing to $\ol{\k}(B)$.

\begin{proof}
We will construct infinitely many Galois $p$-covers over $B$ that remain nonisomorphic after passing to $\ol{\k}$ as varieties over $B_{\ol{\k}}$.
The proof is simpler if we assume that $\k$ contains a primitive $p$-th root of unity, however we do not make this assumption.
In any case, by~\cite[Lemma 16.3.1]{MR2445111},
$\P^1$ admits a Galois cover $\beta\colon C \to \P^1$ of degree $p$
from a geometrically integral smooth curve over $\k$. Let $Z \subset \P^1$ be the branch divisor of $\beta$.

Our goal is to
construct a smooth projective variety $B'$ birational to $B$, a collection of surjective morphisms 
$$
\rho_t \colon B' \to \P^1,
$$
parameterized by elements 
$t$ of an infinite set $U$ 
and a Cartesian diagram
\[
\xymatrix{
\wt{B}_t \ar[d]_{\alpha_t} \ar[r] & C \ar[d]^\beta \\
B' \ar[r]^{\rho_t} & \P^1 \\
}
\]
satisfying the following properties:
\begin{enumerate}
    \item[(a)] $\wt{B}_t$ smooth projective
    and geometrically integral;
    \item[(b)] the branch divisors $D_t = \rho_t^{-1}(Z) \subset B'$ of $\alpha_t$ are pairwise distinct.
\end{enumerate}
Once these conditions are satisfied, we can
take the infinite family of degree $p$ Galois extensions $\{\k(\wt{B}_t)/\k(B)\}_{t \in U}$. By condition (a), we get field extensions $\ol{\k}(\wt{B}_t)/\ol{\k}(B)$.
Let us show that these field extensions are pairwise non isomorphic. If $\ol{\k}(\wt{B}_t)$
and $\ol{\k}(\wt{B}_{t'})$ were isomorphic as field extensions of $\ol{\k}({B'})$, then since both $\wt{B}_{t,\ol{\k}}$ and $\wt{B}_{t',\ol{\k}}$ are normal and finite over $B'_{\ol{\k}}$, they are isomorphic as they coincide with normalization of $B'$ in the same finite extension of $\k(B')$. However $\wt{B}_{t,\ol{\k}}$ and $\wt{B}_{t',\ol{\k}}$ can not be isomorphic over $B'_{\ol{\k}}$ for $t \ne t'$ since they have different branch divisors by condition (b).
Thus by Galois theory we deduce that \eqref{map-GalRes} has infinite image.

Now we construct the collection of morphisms $\rho_t$ satisfying  properties (a) and (b).
We first take $\rho \colon B' \to \P^1$,
the blow up of the base locus of a general very ample pencil $B \dto \P^1$. By construction $B'$ is smooth and projective. If $\dim(B) = 1$ we require in addition that $\rho$ has degree coprime to $p$.
We set $\rho_t = t \circ \rho$, for general $t \in \Aut(\P^1)$.
For condition (b) to be satisfied we can restrict to any dense open subset $U \subset \Aut(\P^1)$ such that for $t, t' \in U$ we have $t't^{-1}(Z) \ne Z$. 

Finally let us explain how we make sure that condition (a) is satisfied. To guarantee that $\wt{B}_t$ is smooth it suffices to require that $t(Z) \subset \P^1$ is disjoint from the closed subset of $\P^1$ parameterizing singular fibers of $\rho$ which is again an open dense condition on $t$. For the fact that $\wt{B}_t$ is geometrically integral we can argue as follows.
If $\dim(B) = 1$, this holds because we required degrees of $\beta$ and $\rho_t$ to be coprime.
On the other hand, if $\dim(B) > 1$, then 
the generic fiber of $\rho_t$ is geometrically integral, 
so the fiber product $\wt{B}_t$ is geometrically integral by~\cite[Exercise 4.3.6]{Liu-AG} using that $\beta$ is flat.
\end{proof}

\section*{Acknowledgments}

We thank 
Jeremy Blanc,
Michel Brion,
Serge Cantat,
Yuri Prokhorov,
Pavel Sechin,
Markus Szymik
for discussions and their interest in our work.
We thank an anonymous referee for pointing out a mistake in Proposition \ref{prop:constr-J} in an earlier version of the paper.
HYL is supported by
the Yushan Fellow Program by the Ministry of Education (NTU-114V1006-5), 
the National Council for Science and Technology in Taiwan (114-2918-I-002 -020-, 114-2628-M-002 -012-), and the Asian Young Scientist Fellowship.
Part of this work was done during his visit to Paris; 
he thanks Claire Voisin and IMJ-PRG for the hospitality, and the
ERC grant HyperK (Grant agreement No. 854361)
for its financial support.
E.S. is supported by the UKRI Horizon Europe guarantee award `Motivic invariants and birational geometry of simple normal crossing degenerations' EP/Z000955/1.

\bibliographystyle{plain}

\end{document}